\documentclass[a4paper,10pt]{article}
\usepackage[english]{babel}
\usepackage{mathtools,amsmath,amssymb,amsfonts,amsthm,mathrsfs}
\usepackage{bbm}
\usepackage[autostyle,threshold=0]{csquotes}
\usepackage{graphicx,color}
\usepackage{tikz,pgfplots,pgf}
\usepackage{epsfig} 
\usepackage{subcaption}
\usepackage[font=small]{caption}

\newtheorem{theorem}{Theorem}[section]
\newtheorem{corollary}[theorem]{Corollary}

\newtheorem{proposition}[theorem]{Proposition}
\theoremstyle{definition}
\newtheorem{definition}[theorem]{Definition}
\theoremstyle{remark}
\newtheorem{remark}[theorem]{Remark}
\newtheorem{example}[theorem]{Example}
\numberwithin{equation}{section}
\newcommand\blfootnote[1]{\begingroup\renewcommand\thefootnote{}\footnote{#1}\addtocounter{footnote}{-1}\endgroup}
\begin{document}
% ------------------------------------------------------------------------
\title{\vspace{-2.0cm}\bf\Large The Ambrosetti-Prodi periodic problem: \\ Different routes to complex dynamics\footnote{This work was performed under the auspices of the Grup\-po Na\-zio\-na\-le per l'Anali\-si Ma\-te\-ma\-ti\-ca, la Pro\-ba\-bi\-li\-t\`{a} e le lo\-ro Appli\-ca\-zio\-ni (GNAMPA) of the Isti\-tu\-to Na\-zio\-na\-le di Al\-ta Ma\-te\-ma\-ti\-ca (INdAM)}
}
\author{{\bf\large Elisa Sovrano and Fabio Zanolin}
\vspace{1mm}\\
{\it\small Department of Mathematics, Computer Science and Physics,}\\
{\it\small University of Udine}\\
{\it\small via delle Scienze 206, 33100 Udine, Italy}\\
{\it\small e-mail: sovrano.elisa@spes.uniud.it}\\
{\it\small e-mail: fabio.zanolin@uniud.it}
}

\date{}

\maketitle

\vspace{-2mm}

%--------------------
%--- Abstract
\begin{abstract}
We consider a second order nonlinear ordinary differential equation of the form $u'' + f(u) = p(t)$ where the forcing term $p(t)$ is a $T$-periodic function and the nonlinearity $f(u)$ satisfies the properties of Ambrosetti-Prodi problems.
We discuss the existence of infinitely many periodic solutions as well as the presence of complex dynamics under different conditions on $p(t)$
and by using different kinds of approaches.
On the one hand, we exploit the Melnikov's method and, on the other hand, the concept of ``topological horseshoe''.
\end{abstract}
%
%
%-----------------------------
%--- Section 1
\section{Introduction}
This paper is devoted to the study of a classical problem in the theory of nonlinear differential equations: the Ambrosetti-Prodi problem.
Motivated by a note appeared in 2011 \cite{Am-11} of Antonio Ambrosetti, in honor of Giovanni Prodi,
we will focus our attention to the case of second order ODEs with periodic coefficients.
Ambrosetti in \cite{Am-11} recalled how, starting with the academic year 1970-1971,
Giovanni Prodi delivered in Pisa a series of lectures on Nonlinear Analysis.
In particular, the first year of the course was devoted to the study of local and global inversion theorems
and the geometry of the infinite dimensional normed spaces.
Indeed, the work made in that period
by Ambrosetti and Prodi on the inversion of functions with singularities
in Banach spaces
led to the publication in 1972 of a seminal paper \cite{AmPr-72} which
can be considered as a milestone since it has influenced the research in
the field of Nonlinear Analysis up to the present days.
The theorems in \cite{AmPr-72} allowed to face new elliptic boundary value problems.
Indeed, the first application concerned the existence and multiplicity of solutions
for a Dirichlet problem with asymmetric nonlinearities
whose derivative crosses the first eigenvalue.
This result received much attention by the mathematical community
and since then problems with these nonlinearities are called ``\emph{Ambrosetti-Prodi problems}'' (briefly written as AP~problems).
It is interesting to observe that the theorems in \cite{AmPr-72} have a very general feature
and therefore they could be applied to different kinds of boundary value problems.

Another meaningful topic pointed out by Ambrosetti in \cite{Am-11} is a list of open questions regarding global inversion theorems
and their applications. With this respect, this work is inspired by one of these:
\blockquote{``To study the periodic case for an ordinary differential second order equation $-u''=g(u)+h(t)$
where $g$ satisfies \[-\infty<\lim_{u\to-\infty}\frac{g(u)}{u}<\lambda_{1}<\lim_{u\to+\infty}\frac{g(u)}{u}<\lambda_{2}\]
and $h$ is $T$-periodic
(this problem was posed by Prodi himself in his course).''\footnote{``\emph{Studiare il caso periodico
per un'equazione ordinaria del secondo ordine con $g$ verificante $(g.1)$ ed $h$ periodica (questo problema fu posto dallo stesso Prodi nel suo corso).}'' Taken from \cite[p.~13]{Am-11}.}}

As observed by Ambrosetti, the AP~problem in a periodic setting seems an interesting issue to be studied further.
It is to notice that several significant results have been already achieved for the periodic case and they concern the existence,
the multiplicity and the stability of periodic solutions (see \cite{DPMaMu-92, FaMaNk-86, NjOm-03, Or-90}).
Accordingly, here we shall focus on some aspects of the AP~periodic problem which, in our opinion,
has not yet been discussed in detail. Our purpose is to prove that second order ODEs,
with periodic coefficients and nonlinearities satisfying conditions of Ambrosetti-Prodi type,
may present solutions with a very complicated behavior as the phase portrait in Fig.~\ref{fig-0} suggests.

\begin{figure}[htb]
          \centering
          \includegraphics[width=0.7\textwidth]{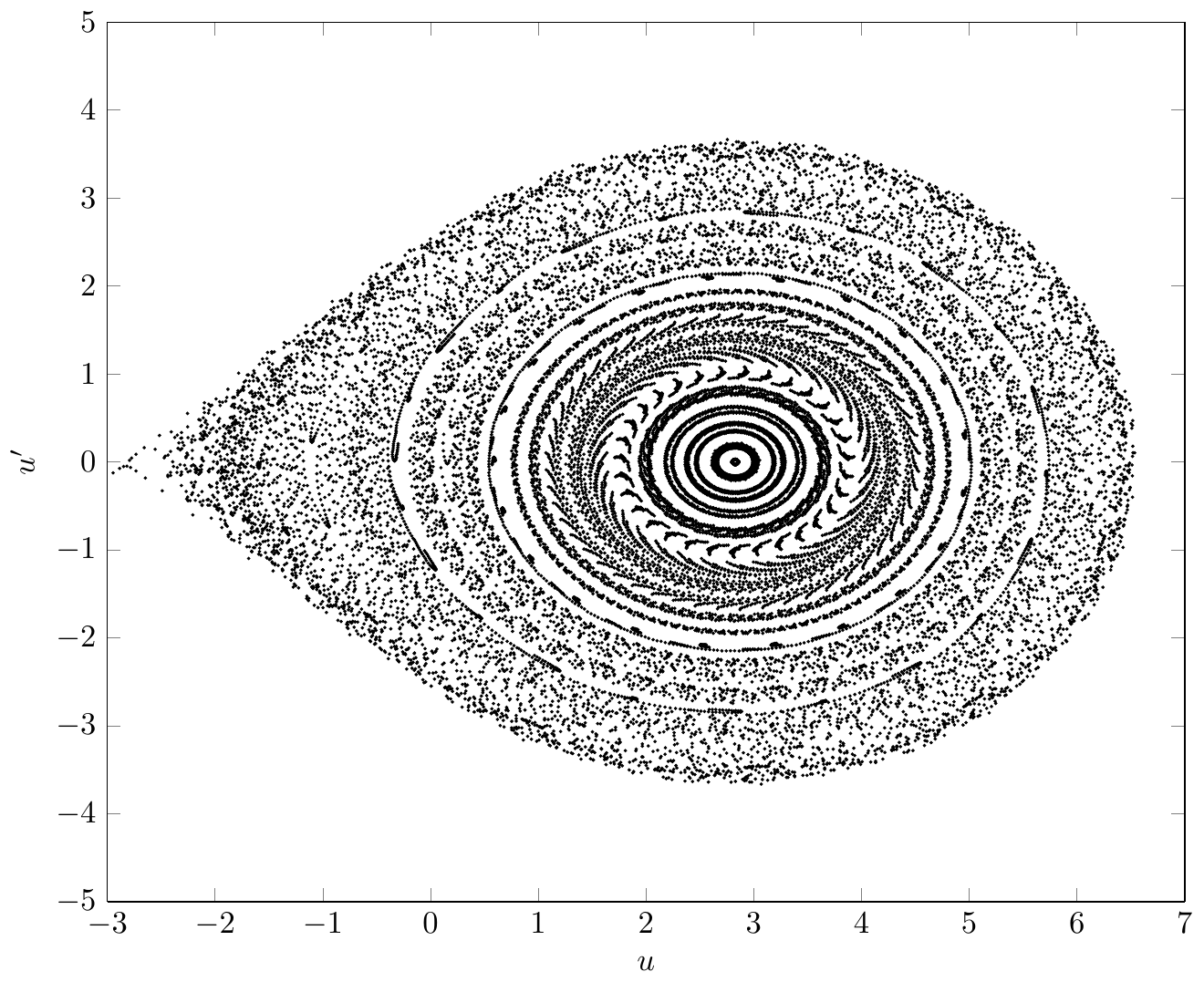}
\caption{Evolution of $u''+\sqrt{1+u^{2}} -1= 2+\varepsilon\sin(\omega t)$
in terms of the iterates of the Poincar\'e map, with $\varepsilon=0.01$, $\omega=10$ and varying 500
initial conditions $(u(0),u'(0))$ where $u(0)$ is within the interval $[-4,6]$ and $u'(0)=0$. The figure displays the
typical alternation of regions of stability and instability or randomness which are common in Hamiltonian systems (cf.~\cite{Mo-73}).}
\label{fig-0}
\end{figure}

To begin, we shall take a historical viewpoint and, also, we shall give some preliminary definitions, in order to clarify the issues we are going to study.

%------------------------
%--- Paragraph 1
\bigskip
\paragraph{The Dirichlet AP problem}{The AP~problem considered in \cite{Am-11,AmPr-72,AmPr-93} involves the following Dirichlet boundary value problem:
\begin{equation}\label{eq-0.1}
\begin{cases}
-\Delta u = f(u) + g(x) &\text{in }\Omega,\\
u=0             &\text{on }\partial\Omega,
\end{cases}
\end{equation}
where $\Omega$ is a bounded open set sufficiently smooth, $g\in C^{0,\alpha}(\bar{\Omega})$ with $\alpha$ a fixed number on $]0,1[$ and $f\in C^{2}(\mathbb{R})$ is a suitable asymmetric function, whose derivative crosses the first eigenvalue associated with the Laplacian operator with zero Dirichlet boundary condition as $s$ goes from $-\infty$ to $+\infty.$ In particular, $f$ is a strictly convex real valued function such that $f(0)=0$ and
\begin{equation}\label{eq-0.2}
0<f'(-\infty)<\lambda_1^{\Omega}(-\Delta)<f'(+\infty)<\lambda_2^{\Omega}(-\Delta),
\end{equation}
where, by definition, $f'(-\infty)=\displaystyle{\lim_{s\to-\infty}f'(s)}$, $f'(+\infty)=\displaystyle{\lim_{s\to+\infty}f'(s)}$ and $\lambda_1^{\Omega}(-\Delta)$, $\lambda_2^{\Omega}(-\Delta)$ are the first two eigenvalues of the eigenvalue problem:
\[\begin{cases}
-\Delta u = \lambda u   &\text{in }\Omega,\\
u=0             &\text{on }\partial\Omega.
\end{cases}
\]
When a function $f$ satisfies these conditions, it is common to said that the problem \eqref{eq-0.1} is of Ambrosetti-Prodi type.
Problems like this one are also called problems with jumping nonlinearities,
in order to stress the existence of two different asymptotes at $\pm\infty$ (cf.~\cite{Fu-75,Fu-76}).
The main result in \cite{AmPr-72} provides the existence of a $C^{1}$
manifold $\mathcal{M}$ of codimension one which separates $C^{0,\alpha}(\bar{\Omega})$
into two disjoint open regions $A_{1}$ and $A_{2}\,$: $C^{0,\alpha}(\bar{\Omega})=A_{1}\cup \mathcal{M}\cup A_{2}$
such that problem \eqref{eq-0.1} has no solutions if $g\in A_{1}$, exactly one solution if $g\in\mathcal{M}$, exactly two solutions if $g\in A_{2}$.

In \cite{ChHa-82}, the same approach is used to give a detailed discussion about two-point boundary value problems, that are associated with nonlinear ODEs of Ambrosetti-Prodi type, where, obviously, conditions in \eqref{eq-0.2} become
\[0<f'(-\infty)<\left(\frac{\pi}{T}\right)^{2}<f'(+\infty)<\left(\frac{2\pi}{T}\right)^{2}.\]

As a result of \cite{AmPr-72}, a very large amount of papers was written on the number of solutions of boundary value problems where the derivative of the nonlinearity jumps the first eigenvalue $\lambda_1^{\Omega}(-\Delta)$ or higher eigenvalues.
Taking into account the survey by de Figueiredo \cite{DF-80}
and the impossibility to give a complete list of references, not even of the main contributions after 1972,
we limit ourselves to quote here the earliest works in the rich mathematical literature about the AP problems.
Hence, we recall that in 1975, Berger and Podolak wrote the function $g$ as $g(x)=\mu\phi_{1}(x)+h(x)$ where $\phi_{1}$ is the first positive eigenfunction associated with $\lambda_1^{\Omega}(-\Delta)$ and $\int_{\Omega}\phi_{1}(x)h(x)dx=0$. Thanks to this decomposition, the manifold $\mathcal{M}$ was characterized by the real parameter $\mu$. In \cite{BePo-75}, under these assumptions, they proved that there exists $\hat{\mu}$ such that problem \eqref{eq-0.1} has no solutions, exactly one or two solutions according as $\mu<\hat{\mu}$, $\mu=\hat{\mu}$ or $\mu>\hat{\mu}$.
In the same year, Kazdan and Warner dealt with the problem of Berger and Podolak assuming a more general second order operator with respect to the Laplacian one and weakening the conditions in \eqref{eq-0.2} on $f$. However, in \cite{KaWa-75}, only the existence of at least one solution if $\mu>\hat{\mu}$ and zero solutions if $\mu< \hat{\mu}$ was proved.
The multiplicity result, which is the characteristic feature of the AP problems,
was then obtained in a weaker form, independently, by Dancer \cite{Da-78} in 1978 and by Amann and Hess \cite{AmHe-79} in 1979.
Subsequently, as we can see by the results of Lazer and McKenna \cite{LaMK-81}, Solimini \cite{So-85} and others,
mathematicians began to wonder about what can happen when the nonlinearity jumps more than one eigenvalue
and so higher multiplicity results can be obtained.
Questions of resonance and non-resonance for asymmetric
jumping nonlinearities can nowadays be also interpreted in the light of the interaction with the so-called
Dancer-Fu\v{c}ik spectrum starting with the pioneering works of Dancer \cite{Da-76,Da-76-77} and Fu\v{c}ik \cite{Fu-76} (see \cite{Ma-07}
for a detailed presentation of this topic).
}

%------------------------
%--- Paragraph 2
\bigskip
\paragraph{The periodic AP problem}{In 1986, the description by means of a real parameter about the set of the solutions of PDEs with Dirichlet boundary conditions achieved in \cite{AmHe-79,BePo-75,Da-78,KaWa-75} is reflected in the work \cite{FaMaNk-86}  by Fabry, Mawhin and Nkashama. In \cite{FaMaNk-86} the Li\'enard ODE considered is
\begin{equation}\label{eq-0.3}
u'' + q(u) u' + f(t,u) = \mu,
\end{equation}
where $\mu\in\mathbb{R}$, $q$ and $f$ are continuous functions, $f$ is $2\pi$-periodic in $t$ and
\begin{equation}\label{eq-0.3.1}
\lim_{|s|\to+\infty}f(t,s)=+\infty \quad \text{uniformly in $t.$}
\end{equation}
Under these conditions, there exists $\hat{\mu}\in\mathbb{R}$ such that equation \eqref{eq-0.3} has no $2\pi$-periodic solution if $\mu<\hat{\mu}$, at least one $2\pi$-periodic solution if $\mu=\hat{\mu}$, at least two $2\pi$-periodic solutions if $\mu>\hat{\mu}$ (cf.~\cite[Cor.~1]{FaMaNk-86}).

On the other hand, Ortega in 1989 deals with the equation:
\begin{equation}\label{eq-0.4}
u''+cu'+f(t,u)=\mu,
\end{equation}
where $c>0$, $f$ is a $T$-periodic function in $t$ satisfying \eqref{eq-0.3.1} and it is strictly convex in $u$:
\[\left(\frac{\partial f}{\partial u}(t,u_{1})-\frac{\partial f}{\partial u}(t,u_{2})\right)(u_{1}-u_{2})>0\,,\quad \text{if $u_{1}\not=u_{2}$, $t\in\mathbb{R}$}.
\]
If, moreover, $f$ is bounded below and
\begin{equation}\label{eq-0.5}
\frac{\partial f}{\partial u}(t,+\infty)\leq \left(\frac{\pi}{T}\right)^{2}+\left(\frac{c}{2}\right)^{2},\quad t\in\mathbb{R},
\end{equation}
the description of the set of the periodic solutions and the study of their stability is done in \cite{Or-89}. Indeed, there exists $\hat{\mu}\in\mathbb{R}$ such that, if $\mu>\hat{\mu}$, \eqref{eq-0.4} has exactly two $T$-periodic solutions, one asymptotically stable and another unstable; if $\mu=\hat{\mu}$, \eqref{eq-0.4} has exactly one $T$-periodic which is not asymptotically stable; if $\mu<\hat{\mu}$, every solution of \eqref{eq-0.4} is unbounded (cf. \cite[Th.~2.1]{Or-89}).
At this point a natural question arises, about improving the upper bound in \eqref{eq-0.5} to the constant
$(2\pi/T)^{2}+c^{2}/4$, which, in this context, plays the role of the second eigenvalue.
Indeed, in 1990 Ortega read in the periodic context the result of \cite{AmPr-72} and, furthermore, he carried out an analysis of the stability properties of $T$-periodic solutions. In fact, in \cite{Or-90}, the following equation is considered:
\begin{equation}\label{eq-0.6}
u''+cu'+f(u)=p(t),
\end{equation}
where $p$ is a continuous and $T$-periodic function and the nonlinear term $f$ is of Ambrosetti-Prodi type, in other words, $f\in C^{2}(\mathbb{R})$ satisfies $f''(s)>0$ for each $s\in\mathbb{R}$ and the adapted conditions in \eqref{eq-0.2}, namely
\begin{equation}
-\infty\leq f'(-\infty)<0<f'(+\infty)\leq \left(\frac{2\pi}{T}\right)^{2}+\left(\frac{c}{2}\right)^{2}.
\end{equation}
Under these assumptions, by denoting with $C_{T}$ the set of $T$-periodic and continuous functions endowed with the sup norm, there exists a closed connected $C^{1}$-manifold $\mathcal{M}$ of codimension one in $C_{T}$ such that $C_{T}\setminus\mathcal{M}$ consists exactly of two connected components $A_{1}$ and $A_{2}$ such that equation \eqref{eq-0.6} has no $T$-periodic solution if $p\in A_{1}$, exactly one solution if $p\in\mathcal{M}$, exactly two solutions if $p\in A_{2}$ (cf.~\cite[Th.~0]{Or-90}).

Another contribution to the treatment of AP periodic problems arises from the work \cite{DPMaMu-92} of Del Pino,
Man\'asevich and Murua which generalizes the results in \cite{LaMK-90}.
In \cite{DPMaMu-92} the number of periodic solutions is in relation with the number of eigenvalues jumped by the nonlinearity.

Finally, as a matter of fact, the study of systems with asymmetric nonlinearities
has motivated a great deal of works that investigate the existence and the multiplicity
of periodic and subharmonic solutions
(see \cite{BoZa-13,FoGh-10,NjOm-03,Or-96,Re-97,ReZa-96,Wa-00,ZaZa-05} and the references therein).
A motivation for these researches, besides the connection with the periodic Dancer-Fu\v{c}ik spectrum,
comes from the topics related to the Lazer-McKenna suspension bridge models \cite{LaMK-90}.
}

%------------------------
%--- Paragraph 3
\bigskip
\paragraph{Basics on chaotic dynamics}{
The main part of our paper is devoted to a discussion about the presence of ``chaotic-like'' solutions for the
periodic AP problem. Such a complex behavior for the solutions will be described in terms of the discrete
dynamical system associated with the Poincar\'{e} map of the problem.
Since in the literature there are different notions of chaos, it is important to specify which kind of chaos we refer.
However, it is interesting to observe that, despite the different definitions considered by several authors,
there is a common feature usually associated with the concept of deterministic chaos which is the possibility
to reproduce a coin-tossing sequence by means of the iterates of a given map.

\blockquote{``The laws of chance, with good reason, have traditionally been expressed in terms of flipping a coin.
Guessing whether heads or tails is the outcome of a coin toss is the paradigm of pure chance.'' (Stephen Smale, \cite{Sm-98}).}

\noindent
From this point of view, an abstract scheme often used to describe symbolic dynamics is given by the \textit{shift map}
(also called Bernoulli shift or shift automorphism) on the sets of two-sided sequences of $m$ symbols.
We recall that, given a collection of $m\geq 2$ symbols, namely $\{0,\dots,m-1\},$ we denote by $\Sigma_m:=\{0,\dots,m-1\}^{\mathbb{Z}}$ the set of all two-sided sequences $S=(s_i)_{i\in\mathbb{Z}}$ with $s_i\in \{0,\dots,m-1\}$ for each $i.$
The set $\Sigma_m$ is endowed with a standard metric that makes it a compact space with the product topology.
The shift map $\sigma: \Sigma_m \to \Sigma_m$ is such that $\sigma(S) = S'=(s'_i)_{i\in\mathbb{Z}}$ with $s'_i = s_{i+1}$ for all $i\in {\mathbb Z}.$
Note that the shift map on $\Sigma_m$ is considered as a model for chaotic dynamics because it contains all the features which usually
characterize the concept of chaos as a whole, such as transitivity, density of periodic points, positive topological entropy and
so on (see \cite{AKM-65, De-89, GuHo-83, Wi-03}).

A way to show a possible chaotic behavior for a map $\phi$ on a metric space is to prove the existence of a \textit{compact invariant set} $\Lambda \subseteq X$ and a \textit{continuous and surjective map} $\Pi : \Lambda \to \Sigma_m$ such that $\Pi\circ \phi(w)= \sigma\circ\Pi(w),$ for all $w\in \Lambda.$
When this situation occurs, we say that $\phi$ is \textit{semiconjugate} to the shift map on $m$ symbols.
If the map $\Pi$ is also \textit{one-to-one} we say that $\phi$ is \textit{conjugate} to the shift map on $m$ symbols.
Clearly, in this latter case, the map $\phi$ restricted to $\Lambda$ inherits all the topological properties of the shift map.

Close to the shift map, a prototypical example of chaotic dynamics arises by the geometric
structure associated with the Smale horseshoe. From a historical background,
the Smale horseshoe ``robustly describes the homoclinic dynamics encountered by Poincar\'{e} and
studied by Birkhoff, Cartwright-Littlewood, and Levinson'' (quoting \cite{Sh-05}).
Technically, the Smale's construction involves a planar diffeomorphism, acting on a square,
whose image has the shape of a horseshoe that crosses the square in a suitable manner
(see \cite{Mo-73,Sm-65,Sm-67} for the mathematical details). The Smale horseshoe map presents a
\textit{hyperbolic compact invariant set} on which it is \textit{conjugate} to the shift map on two symbols.
In the sequel, any time we have a map with the same properties (conjugate to $(\Sigma_{m},\sigma)$ for $m\geq 2$),
we will say that
\textit{a Smale horseshoe occurs}.
This is, for instance, the case considered in the frame of Melnikov's theory
where a Smale horseshoe occurs for some iterates of the Poincar\'{e} map as a consequence of the
Smale-Birkhoff theorem. In fact, such theorem
considers a diffeomorphism $\phi$ possessing
a transversal homoclinic point $q$ to a hyperbolic saddle point $p$. Then, for some $N,$
$\phi$ has a hyperbolic invariant set $\Lambda$ on which the $N$-th iterate $\phi^N$
is conjugate to the shift map on two symbols (see \cite{Ho-90}).

Sometimes the detection of a Smale horseshoe may be a difficult task, and so,
it grew up the idea of the possibility to prove some weaker types of chaos,
which nonetheless, are still important in the applications. For this reason, various concepts of chaotic dynamics have been proposed
in different contexts (see \cite{BlCo-92,BuWe-95,KiSt-89}).
Several authors developed a wide range of different techniques of nonlinear analysis leading to the so-called ``topological horseshoes''
\cite{KeYo-01}. Since in this paper we are interested in the study of the periodic AP problem,
we consider a relevant point the achievement of the existence of periodic solutions (possibly subharmonic ones)
as fixed points of the Poincar\'{e} map or of its iterates.
Therefore, we say that \textit{a topological horseshoe occurs}
if there is a \textit{compact invariant set} on which a given map $\phi$ is \textit{semiconjugate} to the shift map on $m\geq 2$ symbols
and, moreover, for each periodic sequence
$S\in \Sigma_m$, there is at least one periodic point $w\in \Lambda$ with the same period and such that $\Pi(w) = S.$
It is important to observe that also in this case the topological entropy is positive.
}

%------------------------
%--- Paragraph 4
\bigskip
\paragraph{Plan of the paper}{
In view of the previous survey about the AP periodic problems, our aim is to focus the attention on a
simple second order nonlinear ODE in which the nonlinearity contains the principal features about
the crossing of the first eigenvalue. Namely, we consider
\begin{equation}\label{eq-0.8}
u''+f(u)=p(t),
\end{equation}
where $p(t)$ is a $T$-periodic forcing term and
the nonlinearity $f(u)$ is a positive function, with global minimum at $u=0$, which is decreasing on
the negative reals and increasing on the positive ones.
The counterpart of \eqref{eq-0.8} with a damping term
\begin{equation*}
u''+c u'+f(u)=p(t),
\end{equation*}
will be studied as well.
Since a dynamical system approach will be adopted,
we briefly present in Section~\ref{section-2} some results
about the autonomous equation $u''+f(u)=k$ where $k\in\mathbb{R}$. In particular, for $k > 0$
the associated planar phase-portrait is that of a local center enclosed by a homoclinic trajectory
of a hyperbolic saddle point. When \eqref{eq-0.8} may be treated as a small perturbation of the associated autonomous system,
such saddle-center geometry suggests to exploit a Melnikov type approach.
On the contrary, when the perturbation is not necessary small, we discuss two other different methods, one coming from the
Conley index theory and borrowed from \cite{GKMO-00,KoMiOk-96} and, the other one, based on a topological
argument called ``SAP''.
Therefore, in Section~\ref{section-3} we discuss the detection of chaos,
under different assumptions on the nonlinearity $f$ and the forcing term $p$
that will suggest in a natural way the choice of the abstract method that we are going to apply.
Furthermore, these different methods and their corresponding results have been divided mainly into two types
according to their capability to ensure the presence of ``Smale horseshoes'' or of ``topological horseshoes''.
Finally, in Appendix we provide the basic theory of the SAP method for completeness.
}

%------------------------------------------------------------------------------
%--- Section 2
\section{Description of the periodic problem}\label{section-2}

As observed in the Introduction, in order to develop a more complete understanding of the Ambrosetti-Prodi problem with periodic boundary conditions, the typical nonlinearities that one has to consider consist of sufficiently smooth strictly convex functions $f: {\mathbb R} \to {\mathbb R}$ such that
\begin{equation}\label{eq-2.0}
f'(-\infty) < 0 < f'(+\infty).
\end{equation}
This clearly is consistent with the requirement that the derivative of the nonlinearity crosses the first eigenvalue,
which is $\lambda_1 = 0$ in the case of the linear periodic eigenvalue problem associated with the differential operator
$u \mapsto - u'' - c u'$ (where $c \in {\mathbb R}$). Any strictly convex function satisfying \eqref{eq-2.0}
is such that $f(\pm\infty) = +\infty$ and it has a unique point of strict absolute minimum $s = s_{m}$.
Without loss of generality (i.e. possibly replacing $f(s)$ with $f(s + s_{m}) - f(s_{m})$),
we can suppose to work with a nonlinear function $f$ having a strict absolute minimum at $s=0$
and such that $f(0) = 0.$

Taking into account these preliminary observations, we are in position to introduce a list of the main assumptions characterizing the class of nonlinearities that we will consider.
These conditions, summarized below in $(H_{0})$, will be tacitly assumed throughout this paper
and they represent the minimum equipment of requirements which are common in all the different approaches we are going to discuss.
Further regularity conditions will also be introduced in the sequel when needed.

\medskip\noindent
\emph{$(H_{0})\;$
$f:\mathbb{R}\to\mathbb{R}$ is a locally Lipschitz continuous function with $f(0)= 0$
which is strictly decreasing on $]-\infty,0]$ and strictly increasing on $[0,+\infty[$ and such that $\displaystyle{\lim_{|s|\to+\infty}f(s)=+\infty}$.}

\medskip\noindent
We deal with the second order nonlinear equation
\begin{equation*}
u'' +f(u)=p(t),
\leqno{(\mathscr{E}_1)}
\end{equation*}
where the forcing term $p: {\mathbb R} \to {\mathbb R}$ is supposed to be a locally integrable $T$-periodic function.
In some cases it will be possible to extend the results for equation $(\mathscr{E}_1)$ to the equation
\begin{equation*}
u''+c u'+f(u)=p(t),
\leqno{(\mathscr{E}_2)}
\end{equation*}
where $c$ is a positive friction coefficient. If $c$ is assumed to be small,
equation $(\mathscr{E}_2)$ can be viewed as a perturbation of the conservative equation
$(\mathscr{E}_1)$.
For the investigation of both $(\mathscr{E}_1)$ and $(\mathscr{E}_2)$ we follow a dynamical system approach
by analyzing the local flow associated with the corresponding systems in the phase plane. In particular,
dealing with $(\mathscr{E}_1)$, we consider the planar Hamiltonian system
\begin{equation*}
\begin{cases}
x'=y,\\
y'=-f(x)+p(t).
\end{cases}
\leqno{(\mathscr{S}_1)}
\end{equation*}
As usual, by the local flow determined by $(\mathscr{S}_1)$ we mean the map $\Psi_{t_0}^{t}$ which associates to any initial point $z_0=(x_0,y_0)
\in {\mathbb R}^2$ the point $\zeta(t),$ where $\zeta(\cdot)$ is the solution of $(\mathscr{S}_1)$ satisfying the initial condition
$\zeta(0) = z_0$ and defined on its maximal interval of existence. In the sequel, when not otherwise specified, we will
take $t_0 = 0$ and we consider the Poincar\'{e} operator $\Psi:= \Psi_{0}^{T}$. The fundamental theory of ODEs guarantees that
$\Psi$ is a homeomorphism defined on an open set $\text{dom}\Psi\subseteq {\mathbb R}^2$.
Similar considerations can be done for system
\begin{equation*}
\begin{cases}
x'=y,\\
y'=- cy -f(x)+p(t).
\end{cases}
\leqno{(\mathscr{S}_2)}
\end{equation*}
which is equivalent to $(\mathscr{E}_2)$.

The study of system $(\mathscr{S}_1)$ should become easier after a preliminary qualitative analysis of the autonomous system with a constant forcing term.
Roughly speaking, this corresponds to the case in which the time variable is ``freezed'' and it will be the object of the next subsection.

%------------------------
%--- Subsection 2.1
\subsection{Phase plane analysis}
Let us introduce a model problem by means of the autonomous ODE
\begin{equation}\label{eq-2.1}
u''+f(u)=k,
\end{equation}
with $k$ a real parameter. The phase plane analysis and geometric considerations give us information about the qualitative behavior of the solutions of \eqref{eq-2.1} and in turn of $(\mathscr{E}_1)$.

With this respect, equation \eqref{eq-2.1} can be written equivalently as a planar system in the phase plane $(x,y)$:
\begin{equation}\label{eq-2.2}
\begin{cases}
x'=y,\\
y'=-f(x)+k.
\end{cases}
\end{equation}

First of all let us find the equilibria of \eqref{eq-2.2}, by solving the equation $f(x)=k$. In view of $\min_{\mathbb R} f= f(0) = 0,$
we consider from now on only the case $k \geq 0.$

If $k=0$, the origin is an unstable equilibrium of the system. In particular, it is the coalescence of a saddle point with a center.
It seems interesting to observe that in literature such a geometry appears in the so called Bogdanov-Takens
bifurcation (see \cite{GuHo-83}). On the other hand, if $k>0$, the properties of the function $f$ lead to the existence of exactly two equilibria. Under the assumption $(H_{0})$ made on $f$, we can define two homeomorphisms
\[\begin{split}
&f_{l}:=\left.f\right|_{]-\infty,0]}:]-\infty,0]\to[0,+\infty[,\\
&f_{r}:=\left.f\right|_{[0,+\infty[}:[0,+\infty[\to[0,+\infty[,
\end{split}\]
such that $f_{l}$ is strictly decreasing and $f_{r}$ is strictly increasing. Therefore, the inverse functions of both $f_{l}$ and $f_{r}$ are well defined and we denote them
by $f^{-1}_{l}$ and $f^{-1}_{r}$, respectively.
By setting
\[\begin{split}
&x_{u}=x_{u}(k):=f^{-1}_{l}(k),\\
&x_{s}=x_{s}(k):=f^{-1}_{r}(k),
\end{split}\]
we have $x_{u}<x_{s}$. The equilibria are the points $(x_{u},0)$
and $(x_{s},0)$ where the first one has got the topological
structure of an unstable saddle and the second one is a stable
center.

The system \eqref{eq-2.2} is a hamiltonian system with total energy given by
\begin{equation}\label{eq-2.4}
E_{k}(x,y):=\frac{1}{2}y^{2}+F(x)-kx,
\end{equation}
where $F$ is defined by
\[F(x):=\int_{0}^{x}f(s)ds.
\]
Notice that $F(\pm\infty) =\pm\infty.$

To describe the associated phase portrait, for each $\rho\in\mathbb{R}$, we define the energy level lines of \eqref{eq-2.2} as follows
\[
\mathcal{L}_{\rho}:=\{(x,y)\in\mathbb{R}^{2}:E_{k}(x,y)=\rho\}.
\]
In order to study the geometry of each $\mathcal{L}_{\rho}$ it is useful to introduce the auxiliary function
\begin{equation}\label{phi}
\Phi_{k}(x):=F(x)-kx.
\end{equation}
Observe that, for each $k > 0,$ the graph of the function $\Phi_{k}$ is that of a $N$-shaped curve passing through the origin with negative slope.

\begin{proposition}\label{prop-2.1}
Let $\Phi_{k}$ be defined as in \eqref{phi} for $k=0$. Then, $\Phi_{0}(x)=\rho$ has a unique solution for every $\rho\in\mathbb{R}$. In particular, the following hold.
\begin{itemize}
\item If $\rho=0$ the solution is $x=0$.
\item If $\rho<0$ we denote it by $x_{*}(\rho)$ and it is such that $x_{*}(\rho)<0$.
\item If $\rho>0$ we denote it by $x^{*}(\rho)$ and it is such that $x^{*}(\rho)>0$.
\end{itemize}
\end{proposition}
\begin{proof}
From the conditions in $(H_{0})$ made on $f$ follows that $\Phi_{0}$ is strictly increasing on $\mathbb{R}$ and also $\Phi_{0}(0)=0$. Hence, thanks to the monotonicity of $\Phi_{0}$, the conclusions follow straightaway.
\end{proof}

\begin{proposition}\label{prop-2.2}
Let $k$ be a fixed positive real number and $\Phi_{k}$ defined as in \eqref{phi}. Then, the following hold.
\begin{itemize}
\item If $\rho=\Phi_{k}(x_{u})$, then $\Phi_{k}(x)=\rho$ has two solutions. One is $x_{u}$ and the other one, denoted by $x_{h}=x_{h}(k)$, is such that $x_{s}< x_{h}$.
\item If $\rho=\Phi_{k}(x_{s})$, then $\Phi_{k}(x)=\rho$ has two solutions. One is $x_{s}$  and the other one, denoted by $x_{*}(\rho)$, is such that $x_{*}(\rho)<x_{s}$.
\item If $\rho>\Phi_{k}(x_{u})$, then $\Phi_{k}(x)=\rho$ has a unique solution, denoted by $x^{*}(\rho)$, and it is such that $x^{*}(\rho)>x_{h}$.
\item If $\Phi_{k}(x_{s})<\rho<\Phi_{k}(x_{u})$, then $\Phi_{k}(x)=\rho$ has three solutions. These solutions, denoted by $x_{*}(\rho)$, $x_{-}(\rho)$ and $x_{+}(\rho)$, are such that $x_{*}(\rho)<x_{u}<x_{-}(\rho)<x_{s}<x_{+}(\rho).$
\item If $\rho<\Phi_{k}(x_{s})$, then $\Phi_{k}(x)=\rho$ has a unique solution, denoted by $x_{*}(\rho)$, and it is such that $x_{*}(\rho)<x_{u}$.
\end{itemize}
\end{proposition}
\begin{proof}
Assumption $(H_{0})$ leads to $\Phi_{k}(0)=0$. By definition of $\Phi_{k}$, its derivative is $\Phi'_{k}(x)=f(x)-k$. Thus, $\lim_{x\to\pm\infty}\Phi'_{k}(x)=+\infty$ because of the condition $\lim_{|x|\to\infty}f(x)=+\infty$. In this way, we have $\Phi_{k}(\pm\infty)=\pm\infty$.
Moreover, $\Phi_{k}$ has exactly two critical points which are the abscissa of the equilibria of system \eqref{eq-2.2}. From the properties $(H_{0})$
assumed on $f$, we deduce that $x_{u}$ is a local maximum and $x_{s}$ is a local minimum. Therefore, it follows that $\Phi_{k}$ is strictly decreasing on $[x_{u},x_{s}]$ and strictly increasing on $]-\infty,x_{u}]$ and $[x_{s},+\infty[$.
Since $0\in]x_{u},x_{s}[$, we have $\Phi_{k}(x_{u})>0>\Phi_{k}(x_{s})$.

So that, if $\rho=\Phi_{k}(x_{u})$, then there exists unique $x_{h}\in]x_{s},+\infty[$ such that $\Phi_{k}(x_{h})=\Phi_{k}(x_{u})$. Analogously, if $\rho=\Phi_{k}(x_{s})$ then there exists unique $x_{*}(\rho)\in]-\infty,x_{u}[$ such that $\Phi_{k}(x_{*}(\rho))=\Phi_{k}(x_{s}).$
Instead, for every $\rho\in]\Phi_{k}(x_{s}),\Phi_{k}(x_{u})[$, there exist $x_{*}(\rho)\in]-\infty,x_{u}[$, $x_{-}(\rho)\in]x_{u},x_{s}[$ and $x_{+}(\rho)\in]x_{s},x_{h}[$ which are zeros of the equation $\Phi_{k}(x)=\rho.$ At last, if $\rho\in]\Phi_{k}(x_{u}),+\infty[$, or $\rho\in]-\infty,\Phi_{k}(x_{s})[$, the equation $\Phi_{k}(x)=\rho$ has exactly one solution $x^{*}(\rho)\in]x_{h},+\infty[$, respectively $x_{*}(\rho)\in]-\infty,x_{u}[$.
\end{proof}

\noindent
An application of Proposition~\ref{prop-2.1} along with Proposition~\ref{prop-2.2} reveals the geometry of the phase portrait associated with system \eqref{eq-2.2} for any given $k\geq 0$.
Examples of phase portraits which mimic the behavior of the solutions of \eqref{eq-2.1} are shown in Fig.~\ref{fig-1}.

\begin{figure}[htb]
          \centering
          \includegraphics[width=1\textwidth]{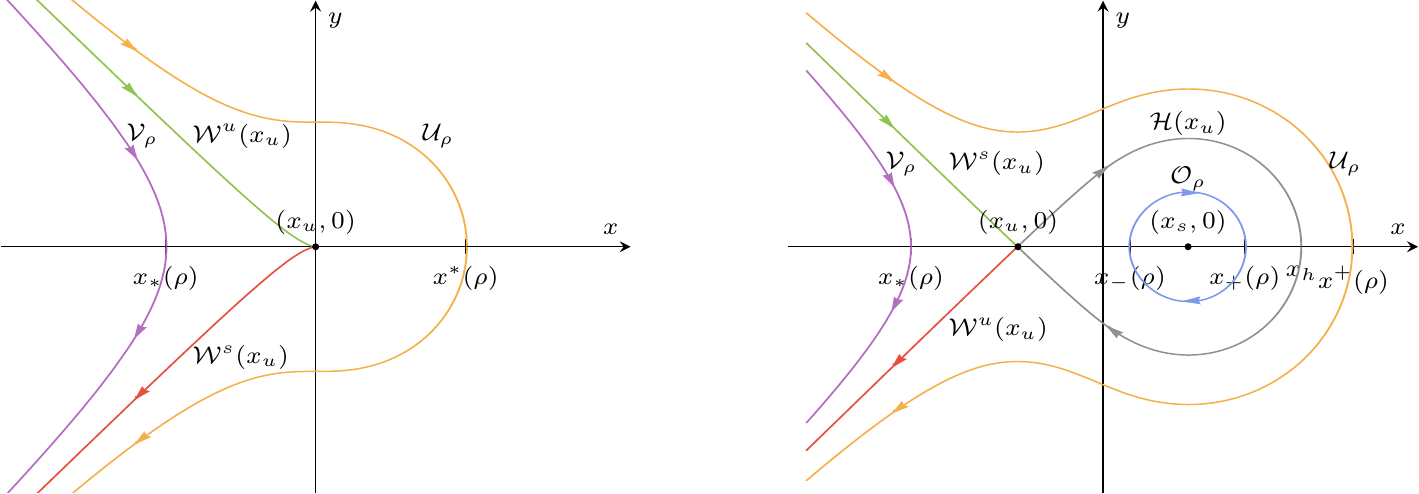}
          \caption{Phase portraits of the autonomous system \eqref{eq-2.2} where the nonlinearity is given by $f(u)=\sqrt{1+u^{2}}-1$. Right: $k=0.$ Left: $k>0$. In both cases the geometry of the different energy level lines is pointed out and the arrows show the direction of the flow along the trajectories.}
\label{fig-1}
\end{figure}

\noindent
Moreover, for all $\rho\in\mathbb{R}$, we can characterize the energy level lines $\mathcal{L}_{\rho}$ according to their type with respect to the level $\rho$. Since, the different kinds of energy level lines for case $k>0$ include the ones for $k=0$, here we give just a detailed discussion about positive reals $k.$

For $\rho=\Phi_{k}(x_{u})$, the saddle like structure is characterized by the union of the unstable equilibrium point with the unstable manifold $W^{u}(x_{u})$, the stable manifold $W^{s}(x_{u})$ and the homoclinic orbit $\mathcal{H}(x_{u})$. In this way, we have
\[\mathcal{L}_{\Phi_{k}(x_{u})}=\{(x_{u},0)\}\cup W^{u}(x_{u})\cup W^{s}(x_{u})\cup \mathcal{H}(x_{u}),
\]
where
\[\begin{split}
&W^{u}(x_{u}):=\{(x,y)\in\mathbb{R}^{2}:x<x_{u},\, y<0,\, E_{k}(x,y)=\Phi_{k}(x_{u})\},\\
&W^{s}(x_{u}):=\{(x,y)\in\mathbb{R}^{2}:x<x_{u}, \, y>0,\, E_{k}(x,y)=\Phi_{k}(x_{u})\},\\
&\mathcal{H}(x_{u}):=\{(x,y)\in\mathbb{R}^{2}:x>x_{u},\, E_{k}(x,y)=\Phi_{k}(x_{u})\}.
\end{split}\]

For $\Phi_{k}(x_{s})<\rho<\Phi_{k}(x_{u})$, the energy level line splits as follows
\[\mathcal{L}_{\rho\in]\Phi_{k}(x_{s}),\Phi_{k}(x_{u})[}=\mathcal{O}_{\rho}\cup \mathcal{V}_{\rho},
\]
where
\begin{equation}\label{eq-2.5o}
\mathcal{O}_{\rho}:=\{(x,y)\in\mathbb{R}^{2}:x>x_{u},\, E_{k}(x,y)=\rho\}
\end{equation}
is a closed symmetric curve surrounding the center which intersects the $x-$axis at the points $(x_{-}(\rho),0)$ and $(x_{+}(\rho),0)$ and it is run in the clockwise sense, on the contrary,
\begin{equation}\label{eq-2.5v}
\mathcal{V}_{\rho}:=\{(x,y)\in\mathbb{R}^{2}:x<x_{u},\, E_{k}(x,y)=\rho\}
\end{equation}
is an unbounded symmetric curve which intersects the $x-$axis at the point $(x_{*}(\rho),0)$.

If $\rho=\Phi_{k}(x_{s})$, then
\[\mathcal{L}_{\Phi_{k}(x_{s})}=\{(x_{s},0)\}\cup \mathcal{V}_{\Phi_{k}(x_{s})},\]
where $\{(x_{s},0)\}$ is the stable equilibrium point and $\mathcal{V}_{\Phi_{k}(x_{s})}$ is defined according to \eqref{eq-2.5v}.

For every $\rho<\Phi_{k}(x_{s})$, $\mathcal{L}_{\rho}$ is a curve identified by \eqref{eq-2.5v} and so, also in this case, we denote each energy level line with $\mathcal{V}_{\rho}$.

For every $\rho>\Phi_{k}(x_{u})$, $\mathcal{L}_{\rho}$ is an unbounded symmetric curve over the saddle like structure which intersects the $x-$axis at the point $(x^{*}(\rho),0)$ and it is run in the clockwise sense. In this case, the energy level line is
\begin{equation}\label{eq-2.5u}
\mathcal{U}_{\rho}:=\mathcal{L}_{\rho\in]\Phi_{k}(x_{u}),+\infty[}=\{(x,y)\in\mathbb{R}^{2}:E_{k}(x,y)=\rho\}.
\end{equation}

We conclude the phase plane analysis performing a study, depending on $k$, of the intersection points between the saddle like structure with the $x-$axis.
\begin{proposition}\label{prop-2.3}
Let $k_{1},k_{2}\in\mathbb{R}$ such that $0\leq k_{1}<k_{2}$ and $\Phi_{k_{1}}, \Phi_{k_{2}}$ defined as in \eqref{phi}, then there exist unique $x_{h}(k_{i})$ for $i\in\{1,2\}$ such that $\Phi_{k_{i}}(x_{u}(k_{i}))=\Phi_{k_{i}}(x_{h}(k_{i}))$ and $x_{u}(k_{2})<x_{u}(k_{1})<x_{h}(k_{1})<x_{h}(k_{2}).$
\end{proposition}

\begin{proof}
From the growth conditions of $f$ in $(H_{0})$ it follows that
\[x_{u}(k_{2})<x_{u}(k_{1})<x_{s}(k_{1})<x_{s}(k_{2}).\]
By the definition of $\Phi_{k}$, we deduce that
\begin{alignat}{2}
&\Phi_{k_{1}}(x)<\Phi_{k_{2}}(x), \quad \forall\, x<0, \label{eq-prop-0} \\
&\Phi_{k_{1}}(x)>\Phi_{k_{2}}(x), \quad  \forall\, x>0. \label{eq-prop-0bis}
\end{alignat}

Since $x_{u}(k_{2})<x_{u}(k_{1})<0$, the condition in \eqref{eq-prop-0} and the fact that $\Phi_{k_{2}}$ is strictly decreasing on $]x_{u}(k_{2}),0]$, imply
\begin{equation}\label{eq-prop-1}
\Phi_{k_{1}}(x_{u}(k_{1}))<\Phi_{k_{2}}(x_{u}(k_{1}))<\Phi_{k_{2}}(x_{u}(k_{2})).
\end{equation}
Thanks to Proposition~\ref{prop-2.2} there exist exactly two positive real numbers $x_{h}(k_{1})$, $x_{h}(k_{2})$ such that $x_{s}(k_{1})<x_{h}(k_{1})$, $x_{s}(k_{2})<x_{h}(k_{2})$ and
\[\Phi_{k_{i}}(x_{u}(k_{i}))=\Phi_{k_{i}}(x_{h}(k_{i})),\quad \text{for } i=1,2.
\]
Using these equalities in \eqref{eq-prop-1} we can get
\begin{equation}\label{eq-prop-2}
\Phi_{k_{1}}(x_{h}(k_{1}))<\Phi_{k_{2}}(x_{h}(k_{2})).
\end{equation}
Whereas $x_{h}(k_{2})>0$, then from the condition in \eqref{eq-prop-0bis} follows
\begin{equation}\label{eq-prop-3}
\Phi_{k_{2}}(x_{h}(k_{1}))<\Phi_{k_{1}}(x_{h}(k_{1})).
\end{equation}
Combining \eqref{eq-prop-2} and \eqref{eq-prop-3}, we obtain $\Phi_{k_{1}}(x_{h}(k_{1}))<\Phi_{k_{1}}(x_{h}(k_{2})).$
Since $\Phi_{k_{1}}$ is strictly increasing on $[x_{s}(k_{1}),+\infty[$, we conclude that
\[x_{h}(k_{1})<x_{h}(k_{2}),\]
because of $x_{s}(k_{1})<x_{h}(k_{2})$.
\end{proof}

%------------------------
%--- Subsection 2.2
\subsection{Time mapping formulas}
Let us introduce some notation that will be used throughout the paper.
Considering \eqref{eq-2.4} and \eqref{phi}, the time needed to a solution to move in the phase plane $(x,y)$
along an orbit path identified by the energy level $\rho$,
from a point $(x_{1},y_{1})$ to a point $(x_{2},y_{2})$, is given by
\begin{equation}\label{eq-2.6}
\tau(\rho;\, x_{1},x_{2}):=\int_{x_{1}}^{x_{2}}\frac{1}{\sqrt{2(\rho-\Phi_{k}(s))}} \,ds.
\end{equation}
The function $\rho\mapsto \tau(\rho;\, x_{1},x_{2})$ is called time-map associated with the autonomous equation \eqref{eq-2.1}.

The phase plane analysis has highlighted the presence of a saddle like structure and also mainly two types of orbits.
More in detail, there are the periodic orbits, $\mathcal{O}_{\rho}$, and
the non-periodic ones, $\mathcal{V}_{\rho}$ and $\mathcal{U}_{\rho}$.
With this in mind, we can characterize the time-map formulas in three different kinds.

In the case of the periodic orbits, by $\eqref{eq-2.6}$ we can evaluate the time elapsed to move along the orbit $\mathcal{O}_{\rho}$ which is defined as in \eqref{eq-2.5o}.
In particular, we set the time needed to travel from $(x_{-}(\rho),0)$ to a point $(r,0)$ on $\mathcal{O}_{\rho}$, with $x_{-}(\rho)<r\leq x_{+}(\rho)$, as follows
\begin{equation}\label{tau-o}
\tau_{\scriptscriptstyle{\mathcal{O}}}(\rho;\, r):=\tau(\rho;\, x_{-}(\rho),r)=\int_{x_{-}(\rho)}^{r}\frac{1}{\sqrt{2(\rho-\Phi_{k}(s))}} \,ds.
\end{equation}
In this way, since $\mathcal{O}_{\rho}$ is a closed symmetric curve, its fundamental period is $2\, \tau_{\scriptscriptstyle{\mathcal{O}}}(\rho;\, x_{+}(\rho))$.

With respect to the non-periodic orbits, firstly we consider the unbounded curve $\mathcal{V}_{\rho}$ defined as in \eqref{eq-2.5v}.
To evaluate the travel time on $\mathcal{V}_{\rho}$, let us fix a value $r$ with $r<x_{*}(\rho)<x_{u}$.
Then, we define two points that belongs to $\mathcal{V}_{\rho}$: one is $P^{+}_{\rho}(r):=\big(r,\sqrt{2(\rho-\Phi_{k}(r))}\big)$, in the upper half plane,
and the other symmetric one is $P^{-}_{\rho}(r):=\big(r,-\sqrt{2(\rho-\Phi_{k}(r))}\big)$, in the lower half plane.
Therefore, the time needed to move along $\mathcal{V}_{\rho}$ from $P^{+}_{\rho}(r)$ to $(x_{*}(\rho),0)$ is
\[\tau(\rho;\, r,x_{*}(\rho))=\int_{r}^{x_{*}(\rho)}\frac{1}{\sqrt{2(\rho-\Phi_{k}(s))}} \,ds,
\]
which is equal to the time needed to travel from $(x_{*}(\rho),0)$ to $P^{-}_{\rho}(r).$
It follows that the time elapsed to go from $P^{+}_{\rho}(r)$ to $P^{-}_{\rho}(r)$ on $\mathcal{V}_{\rho}$ is
\begin{equation}\label{tau-v}
\tau_{\scriptscriptstyle{\mathcal{V}}}(\rho;\, r):=2\int_{r}^{x_{*}(\rho)}\frac{1}{\sqrt{2(\rho-\Phi_{k}(s))}} \,ds.
\end{equation}

In a similar way we face the time-map associated with the orbit $\mathcal{U}_{\rho}$ defined as in \eqref{eq-2.5u}.
In this case, we fix a value $r<x^{*}(\rho)$ and so, as before, the time needed to go from $P^{+}_{\rho}(r)$ to $P^{-}_{\rho}(r)$
along $\mathcal{U}_{\rho}$ is given by
\begin{equation}\label{tau-u}
\tau_{\scriptscriptstyle{\mathcal{U}}}(\rho;\, r):=2\int_{r}^{x^{*}(\rho)}\frac{1}{\sqrt{2(\rho-\Phi_{k}(s))}} \,ds.
\end{equation}

%-------------------------------------------------------------------------------------------
%--- Section 3
\section{Chaotic solutions for the AP periodic problem}\label{section-3}
After these preliminary considerations on the associated planar system with constant coefficients,
we proceed to present some possible approaches showing how the periodic AP problem may have a great amount of periodic solutions as well as chaotic dynamics.
In this section, we firstly give direct applications of some results already available in the literature
and then we provide a new result that involves a method derived from the theory of topological horseshoes.

%------------------------
%--- Subsection 3.1
\subsection{Melnikov type approach}

We collect here some different tools that are all based on a peculiar structure of autonomous Hamiltonian system, namely the existence of a
hyperbolic fixed point connected to itself by a homoclinic orbit.
This feature is owned by system \eqref{eq-2.2} for $k > 0$ provided that $f$ is sufficiently smooth with $f'(x_u) < 0.$
In order to have such a condition satisfied for every possible choice of $k > 0$,
we assume, along this subsection, a more restrictive condition than $(H_0)$ that is the following one.

\medskip\noindent
\emph{$(H_{1})\;$
$f:\mathbb{R}\to\mathbb{R}$ is a strictly convex function of class $C^r$, for $r\geq 2$,
with $f(0)= 0$, $f(s)>0$ for all $s\not=0$ and $\displaystyle{\lim_{|s|\to+\infty}f(s)=+\infty}$.}
\medskip\noindent

The phase plane analysis shows the presence of an equilibrium point $A = A_k:=(x_s,0)$,
which is a center, and a hyperbolic saddle equilibrium point $B= B_k:= (x_u,0)$
with a homoclinic orbit $\mathcal{H}=\mathcal{H}(x_u)$ enclosing $A.$
As mentioned above, this is the
classical scheme considered in the Melnikov's theory, where system $(\mathscr{S}_1)$
can be viewed as a perturbation of the autonomous
system \eqref{eq-2.2}. The Melnikov's method is a powerful tool
and ``one of the few analytical methods available for the
detection and the study of chaotic motions'' (quoting \cite[p.~186]{GuHo-83}).
In the special case of $(\mathscr{S}_1)$ it can be
applied by splitting the forcing term $p(t)$ as
\begin{equation}\label{eq-3.1}
p(t) = k + \varepsilon p_0(t), \quad k > 0.
\end{equation}
Without loss of generality, we can also suppose that $p_0(t)$ changes its sign. In particular,
by transferring the mean value of $p_0$ to the constant $k$, we can assume
\begin{equation}\label{eq-3.2}
\int_0^T p_0(t)\,dt = 0.
\end{equation}
Let $q(t) = q_{k}(t)$ be the solution of equation \eqref{eq-2.1} such that $q(0) = x_h$ and $q'(0)= 0.$
Recall that $x_h$ (depending on the coefficient $k$) is the solution of $\Phi_{k}(x) = \Phi_{k}(x_u)$
with $x > x_s$ or, equivalently, the point $(x_h,0)$ is the intersection of the homoclinic trajectory $\mathcal{H}$
with the $x-$axis. The curve $t\mapsto (q(t),q'(t))$ is a particular parametrization of $\mathcal{H}$
and it is unique up to a shift in the time variable. Our choice, which is the standard one in similar situations, is convenient because $q(t)$ is an even function.
Moreover, by standard results on hyperbolic saddle points, note that $|q(t)- x_u| + |q'(t)| \to 0$ with exponential decay as $t\to \pm\infty$ (cf. \cite[Ch. III.6]{Ha-80}).
Thus, in particular, the improper integrals $\int_0^{+\infty} (q(t) - x_u)\,dt$ and  $\int_0^{+\infty} |q'(t)|\,dt$
are convergent.

Now, the Melnikov function associated with system $(\mathscr{S}_1)$ for $p(t)$ as in
\eqref{eq-3.1}, is given by
\begin{equation}\label{eq-3.3}
\Delta(\alpha):= \int_{-\infty}^{+\infty} q'(t)p_0(t+\alpha)\,dt.
\end{equation}
Notice that, by the $T$-periodicity of $p_0(t)$, it turns out that also $\Delta(\alpha)$ is a $T$-periodic function.
Moreover, from \eqref{eq-3.2} we have $\int_{0}^T \Delta(\alpha)\,d\alpha =0,$
so that either $\Delta\equiv 0$ or $\Delta(\alpha)$ changes its sign.

An application of the Melnikov method to system
\begin{equation}\label{eq-3.4}
\begin{cases}
x'=y,\\
y'=-f(x)+k + \varepsilon p_0(t),
\end{cases}
\end{equation}
gives the following result (cf.~\cite[Th.~4.5.3]{GuHo-83} or \cite[Th.~28.1.7]{Wi-03}).
\begin{theorem}\label{th-3.1} Assume $(H_{1})$ and let $(q(t),q'(t))$ be the homoclinic solution at the saddle point $B=B_k$
for the autonomous system \eqref{eq-2.2} for some $k > 0.$ Let also $p_0$ be a sufficiently smooth, $C^r$ for $r\geq 2$,
$T$-periodic function satisfying $\eqref{eq-3.2}$. If there exists $\alpha \in [0,T[$ such that $\Delta(\alpha) = 0$ and $\Delta'(\alpha) \not=0,$
then there is $\varepsilon_0 > 0$ such that for each $\varepsilon$ with $0 < |\varepsilon| < \varepsilon_0$
a Smale horseshoe occurs for some iterate of the Poincar\'{e} map associated with system \eqref{eq-3.4}.
\end{theorem}
The result expressed in Theorem~\ref{th-3.1} is robust for small smooth perturbations. More in detail, the presence of a Smale horseshoe
is guaranteed also for system
\[\begin{cases}
x'=y,\\
y'=-cy -f(x)+k + \varepsilon p_0(t),
\end{cases}\]
provided that $c$ is sufficiently small, depending on $\varepsilon.$ Hence the result applies to
equation
\[u'' + cu' + f(u) = k + \varepsilon p_0(t)
\]
as well. More precisely, if we write the coefficient $c$ as
\[c:= \varepsilon c_0\,,\]
the Melnikov function takes the form
\begin{equation}\label{eq-3.3b}
\Delta(\alpha):= \int_{-\infty}^{+\infty} \left( q'(t)p_0(t+\alpha) - c_0 q'(t)^2\right)\,dt
\end{equation}
and Theorem~\ref{th-3.1} applies to system
\[\begin{cases}
x'=y,\\
y'=-\varepsilon c_0y -f(x)+k + \varepsilon p_0(t).
\end{cases}\]

Usually the test of the existence of a simple zero for the Melnikov function is a hard task, especially
if an explicit analytical expression for $q(t)$ is not given. The first important and pioneering
applications of this method to some second order nonlinear ODEs, such as the pendulum or the Duffing equation,
have taken advantage of the fact that the expression of $q(t)$ was known (see \cite[p.~191]{GuHo-83}).

On the contrary, when an explicit expression of $q(t)$  is not given, some results can be still produced by exploiting further qualitative information
about the homoclinic orbit or even about the forcing term, if they are available.
From this point of view, we refer to the work \cite{BaFe-02b} of Battelli and Fe\v{c}kan since they have evaluated the Melnikov function
when $q(t)$ is a rational function of $\exp(t).$
A general result, which does not require any specific assumption on $q(t)$ by involving only a simply verifiable condition on $p_0(t)$,
was obtained by Battelli and Palmer in \cite{BaPa-93}. This result applies to system \eqref{eq-3.4} provided that the period of the forcing term is sufficiently large.
For this reason, instead of \eqref{eq-3.4}, it is convenient to consider the system
\begin{equation}\label{eq-3.5e}
\begin{cases}
x'=y,\\
y'=-f(x)+k + \varepsilon^2 p_0(\varepsilon t).
\end{cases}
\end{equation}
In this setting, we can state what follows (cf.~\cite[p.~293, Theorem]{BaPa-93}).
\begin{theorem}\label{th-3.2} Assume $(H_{1})$ with $f\in C^{r+3}$, for $r\geq 5$,
and let $(q(t),q'(t))$ be the homoclinic solution at the saddle point $B=B_k$ for the autonomous system \eqref{eq-2.2} for some $k > 0$.
Let also $p_0$ be a sufficiently smooth, $C^{r+3}$ for $r\geq 5$, $T$-periodic function satisfying $\eqref{eq-3.2}$.
If there exists $\alpha \in [0,T[$ such that
\[p_0'(\alpha)= 0\not= p_0''(\alpha),\]
then there is $\varepsilon_0 > 0$ such that for each $\varepsilon$ with $0 < |\varepsilon| < \varepsilon_0$
a Smale horseshoe occurs for some iterate of the Poincar\'{e} map associated with system \eqref{eq-3.5e}.
\end{theorem}

We conclude this section by giving an example of application of these results to the AP periodic problem.
In order to do this, given $\omega > 0$, we suppose that $p_0(t):= \sin(\omega t)$ is the periodic forcing term of period $T:=2\pi/\omega.$
Using the properties of $p_{0}(t)$ we can show the following consequence.

\begin{corollary}\label{cor-3.1}
Assume $(H_{1})$ and let $(q(t),q'(t))$ be the homoclinic solution at the saddle point $B=B_k$
for the autonomous system \eqref{eq-2.2} for some $k > 0.$
Then, for any $\omega > 0$ there exists $\varepsilon_0 = \varepsilon_0(\omega) > 0$
such that for each $\varepsilon$ with $0 < |\varepsilon| < \varepsilon_0$
a Smale horseshoe occurs for some iterate of the Poincar\'{e} map associated with system
\begin{equation}\label{eq-3.4om}
\begin{cases}
x'=y,\\
y'=-f(x)+k + \varepsilon \sin(\omega t).
\end{cases}
\end{equation}
The same result also holds for the damped system
\begin{equation}\label{eq-3.4omc}
\begin{cases}
x'=y,\\
y'=-\varepsilon c_0 y -f(x)+k + \varepsilon \sin(\omega t),
\end{cases}
\end{equation}
for $c_0$ sufficiently small.
\end{corollary}

\begin{proof}
For simplicity, we investigate only the frictionless case
because with a similar argument one can also derive the result when a small friction term $c_{0}$ is present.

Recalling that $q'(t)$ is an odd function, from \eqref{eq-3.3} we obtain
\begin{equation*}
\Delta(\alpha)= \int_{-\infty}^{+\infty} q'(t)\sin(\omega t+\omega \alpha)\,dt
= - 2\omega \cos(\omega\alpha)\eta(\omega),
\end{equation*}
for
\begin{equation*}
\eta(\omega):=\int_{0}^{+\infty} \tilde{q}(t)\cos(\omega t)\,dt,\quad\text{with }\;  \tilde{q}(t):= q(t) - x_u\,.
\end{equation*}
In this manner, we have reduced the search of a simple zero for $\Delta(\alpha)$ to the verification that
$\eta(\omega)\not=0.$

Since $\eta'(\omega) = -\omega \int_{0}^{+\infty}
\tilde{q}(t)\sin(\omega t)\,dt = - \int_{0}^{+\infty} \tilde{q}(\xi/\omega)\sin(\xi)\,d\xi$, we find that
\begin{equation*}
-\eta'(\omega) = \sum_{j=0}^{\infty} (-1)^j\int_{j\pi}^{(j+1)\pi} \tilde{q}\left(\dfrac{\xi}{\omega}\right)|\sin(\xi)|\,d\xi\,
= \sum_{j=0}^{\infty} (-1)^j\Xi_j
\end{equation*}
where we have set
\begin{equation*}
\Xi_j:=\int_{0}^{\pi} \tilde{q}\left(\dfrac{t + j\pi}{\omega}\right)\sin(t)\,dt.
\end{equation*}
By observing that $\tilde{q}(t)$ is positive and decreasing on $[0,+\infty[$, follows that the sequence
$(\Xi_j)_{j}$ is positive, decreasing and $\Xi_j\to 0$ as $j\to +\infty.$ The theory of alternating series
guarantees that $\sum (-1)^j \Xi_j > 0$ and hence $\eta'(\omega) < 0$ for each $\omega > 0.$ Since $\eta(\omega) \to
\int_{0}^{+\infty} \tilde{q}(t)\,dt > 0 $ as $\omega\to 0^+,$ we conclude that
either $\eta(\omega) > 0$ for each $\omega > 0$ or $\eta(\omega)$ vanishes exactly once.
On the other hand, by the Riemann-Lebesgue lemma, it follows that $\eta(\omega) \to 0$ as $\omega\to +\infty.$ This implies that the second
alternative never occurs because $\eta$ is strictly decreasing. Hence, in view of Theorem~\ref{th-3.1} the proof is completed.
\end{proof}

\begin{remark}
It is interesting to observe that Corollary~\ref{cor-3.1} is applicable to nonlinearities that satisfy the assumptions,
interpreted in the context of periodic problems, which were made by Ambrosetti and Prodi.
In fact, according to \cite[p.~239]{AmPr-72},
let us consider a function $f: {\mathbb R} \to {\mathbb R}$ of class $C^2$ satisfying:
$f(0) = 0,$ $f''(s) > 0$ for all $s\in {\mathbb R}$ and
\begin{equation*}
 \ell':=\lim_{s\to -\infty} f'(s) < \lambda _1 < \ell'' :=\lim_{s\to +\infty } f'(s)< \lambda_2\,.
\end{equation*}
We recall that in the context of a periodic problem the first two eigenvalues associated
with the differential operator $u\mapsto -u''$ subject to $T$-periodic boundary conditions are $\lambda_{1} := 0$ and $\lambda_{2}:= \omega^2$
where $\omega:=(2\pi/T)$.

An application of the Ambrosetti-Prodi abstract theory \cite{AmPr-72} to this situation, yields to the
following result: \textit{there exists $k_0 = k_0(\varepsilon)$ such that, the $T$-periodic problem associated with equation \eqref{eq-3.4om},
has no solutions, exactly one or two solutions according as $k<k_0$, $k=k_0$ or $k> k_0$.} In \cite{Or-89, Or-90}, Ortega
has proved that this result is still valid for the equation $u'' + cu' + f(u) = k + p(t)$, for $c > 0$,
and information about the stability of the solutions are given. In particular,
if we apply such results to system \eqref{eq-3.4omc} we know that if $\ell'' \leq \omega^2/4,$  then, for $k > k_0$,
one $T$-periodic solution is unstable, while the other one is asymptotically stable (cf. \cite[Th.~2.1]{Or-89}).

Since Corollary~\ref{cor-3.1} can be applied without any restriction on $\omega,$
we have that chaos coexists in the same range of parameters where both the theorem of Ambrosetti-Prodi and the ones of Ortega are valid.
Rather surprisingly (at first glance) however there is no real conflict between these results.
Indeed, Melnikov's method ensures the existence of a Smale horseshoe for a suitable iterate, $\Psi^N$, of the associated Poincar\'{e} map;
in particular, this fact also implies the existence of large order subharmonics.
On the other hand, the results in \cite{Or-89,Or-90} prevent the existence of subharmonics of order two.
The existence of a great amount of subharmonic solutions for the AP periodic problem has already been obtained in
\cite{BoZa-13, Re-97, ReZa-96} for the Hamiltonian case, i.e. system $(\mathscr{S}_1)$,
using the Poincar\'{e}-Birkhoff twist theorem (see also \cite{DPMaMu-92, LaMK-87, LaMK-90} for some previous relevant contributions in this direction).
In any case, the coexistence of stability regions and chaos zones is a well-known fact in the theory of  Hamiltonian systems (see \cite[Ch.~III]{Mo-73}).
\end{remark}

The advantage of Corollary~\ref{cor-3.1} is that no condition on $\omega$, and thus on the period $T$, is required.
Nevertheless, this result applies to a limited class of forcing terms.
To achieve an analogous goal for a broad family of periodic functions $p_0,$ we should look at Theorem~\ref{th-3.2}.
In this case, however, we have to take into account the fact that the period of the forcing term is modified by the parameter $\varepsilon > 0$.
In fact, if $p_{0}$ is $T$-periodic, then the forcing term in \eqref{eq-3.5e} has period $T_{\varepsilon}:=T/\varepsilon$.
Observe that in this case, the second eigenvalue of the corresponding periodic problem becomes $\lambda_2 := (2\pi/T)^2\varepsilon^2.$
Hence, for a sufficiently small $\varepsilon > 0$, we will have $f'(+\infty) > \lambda_2$.
This means that the nonlinearity jumps certainly the second eigenvalue (and maybe it crosses many others more),
therefore we enter in a range of parameters for which several $T_{\varepsilon}$-periodic solutions exist.
Indeed, we know that at least the Hamiltonian system $(\mathscr{S}_1)$ has plenty of periodic solution (see \cite{LaMK-87, LaMK-90, Re-97, Wa-00, ZaZa-05}).
Hence, it is reasonable to expect to find also chaotic-like solutions for forcing terms which are not necessarily small. This will be discussed in the sequel.

To conclude this section, based on the applications of the Melnikov's method, we remark that results ensuring the chaotic behavior
of the Poincar\'{e} map, and not just its suitable iterate, can be found in the context of singular perturbation systems in \cite{Ch-99}.
These results, if applied to our case, lead to the introduction of further conditions on the nonlinearity or the period. 

%------------------------
%--- Subsection 3.2
\subsection{Topological horseshoes}
We have already seen that check the hypothesis on the simplicity of the zero for the Melnikov function
may be very laborious when an explicit analytical expression of the homoclinic solution
is not available. Consequently, we discuss two different approaches connected with the Melnikov's theory
and another one which is called stretching along the paths method (or briefly SAP method).
Here weaker requirements are made in order to achieve the presence of a \textit{topological horseshoe}, instead of a Smale horseshoe.

\subsubsection{Slowly varying systems}
In \cite{BaFe-02a}, Battelli and Fe\v{c}kan have generalized the hypothesis about the existence of a simple zero for the
Melnikov function using topological degree and by simply assuming that $\Delta(\alpha)$ changes it sign. This weaker condition, clearly
enlarges the range of applicability of the theorem and the corresponding result reads as follows
(cf.~\cite[Theorem~4.4 and Remark~5.4]{BaFe-02a}).
\begin{theorem}\label{th-3.3}
Assume $(H_{1})$ and let $(q(t),q'(t))$ be the homoclinic solution at the saddle point $B=B_k$
for the autonomous system \eqref{eq-2.2} for some $k > 0$. Let also $p_0$ be a sufficiently smooth, $C^{r}$ for $r\geq 2$,
$T$-periodic function satisfying $\eqref{eq-3.2}$. If
\[\Delta\not\equiv 0,\]
then there is $\varepsilon_0 > 0$ such that for each $\varepsilon$ with $0 < |\varepsilon| < \varepsilon_0$
a topological horseshoe occurs for some iterate of the Poincar\'{e} map associated with \eqref{eq-3.4}.
\end{theorem}

To show an application of Theorem~\ref{th-3.3}, we consider the system
\begin{equation}\label{eq-3.4e}
\begin{cases}
x'=y,\\
y'=-f(x)+k + \varepsilon p_0(\Omega t),
\end{cases}
\end{equation}
where $p_0:{\mathbb R} \to {\mathbb R}$ is a $T$-periodic function of class $C^2$ and $\Omega > 0$ is a fixed constant.
Clearly, this is an example of system $(\mathscr{S}_1)$ with a periodic forcing term $p(t) = k + \varepsilon p_0(\Omega t)$
of period $T_{\Omega}:= T/\Omega.$ In this context, we prove what follows.

\begin{corollary}\label{cor-3.2}
Assume $(H_{1})$ and let $(q(t),q'(t))$ be the homoclinic solution at the saddle point $B=B_k$
for the autonomous system \eqref{eq-2.2} for some $k > 0.$ Suppose that $p_{0}$ is not constant.
Then, there exists $\Omega_0 > 0$ such that for every $\Omega$ with $0 < \Omega < \Omega_0,$
there is  $\varepsilon_0 = \varepsilon_0(\Omega) > 0$ such that for each $\varepsilon$ with $0 < |\varepsilon| < \varepsilon_0$
a topological horseshoe occurs for some iterate of the Poincar\'{e} map associated with system \eqref{eq-3.4e}.
The same result also holds for the damped system
\begin{equation}\label{eq-3.4f}
\begin{cases}
x'=y,\\
y'=- \varepsilon c_0 y -f(x)+k + \varepsilon p_0(\Omega t),
\end{cases}
\end{equation}
for $c_0$ sufficiently small.
\end{corollary}

\begin{proof}
We prove the statement for equation \eqref{eq-3.4e}, since the corresponding conclusion for \eqref{eq-3.4f} holds
as a consequence of the fact that the result in Theorem~\ref{th-3.3} is stable for small perturbations, being based on topological degree theory.
Moreover, note that the same analysis we are going to present can be repeated, in the latter case, for the Melnikov function in \eqref{eq-3.3b}.

Therefore, let us consider \eqref{eq-3.4e}. After an integration by parts, the Melnikov function defined in \eqref{eq-3.3} takes the form
\[\Delta(\alpha)=  - \Omega \int_{-\infty}^{+\infty}  \tilde{q}(t)p_0'(\Omega t+\Omega \alpha)\,dt,
\]
where $\tilde{q}(t) = q(t) - x_u\,.$ Since $p_0$ is not constant, there exists $s^{*}$ such that $p'_0(s^{*}) > 0.$
Then, there exist a constant $\delta^{*} > 0$ and an interval $[s^{*}- r^{*}, s^{*}+r^{*}]$ such that $p'(\xi) \geq \delta^{*}$ for all
$\xi \in [s^{*}- r^{*}, s^{*}+r^{*}]$. Taking $\alpha^{*} = \alpha^{*}(\Omega):= s^{*}/\Omega,$ we have that
\[\begin{split}
- \frac{\Delta(\alpha^{*})}{\Omega} &\geq  \int_{-r^{*}/\Omega}^{r^{*}/\Omega} \tilde{q}(t) p_{0}'(s^{*}+\Omega t)\,dt
-2 \|p_{0}'\|_{\infty}\,\int_{r^{*}/\Omega}^{+\infty} \tilde{q}(t)\,dt\\
 &\geq 2\delta^{*} \int_{0}^{r^{*}/\Omega}\tilde{q}(t)\,dt - 2 \|p_{0}'\|_{\infty}\,\int_{r^{*}/\Omega}^{+\infty} \tilde{q}(t)\,dt.
\end{split}\]
Since, one can deduce the existence of a constant $\Omega_1 > 0$ such that for each $\Omega$ with $0 < \Omega < \Omega_1$ it holds that
$\int_{0}^{r^{*}/\Omega}\tilde{q}(t)\,dt  > (\delta^{*})^{-1}\|p_{0}'\|_{\infty} \,\int_{r^{*}/\Omega}^{+\infty} \tilde{q}(t)\,dt$, then we have $\Delta(\alpha^{*}) < 0.$

Similarly, there exists $s_{*}$ such that $p'_{0}(s_{*}) < 0$. Accordingly, there are a constant $\delta_{*} > 0$ and an interval $[s_{*}- r_{*}, s_{*}+r_{*}]$
such that $p'(\xi) \leq - \delta_{*}$ for all $\xi \in [s_{*}- r_{*}, s_{*}+r_{*}]$. Taking now $\alpha_{*} = \alpha_{*}(\Omega):= s_{*}/\Omega,$
by an argument similar to the previous one, there exists a constant $\Omega_{2} > 0$
such that for each $\Omega$ with $0 < \Omega < \Omega_{2}$ we have $\Delta(\alpha_{*}) > 0.$
The conclusion now follows from Theorem~\ref{th-3.3} by taking $\Omega_{0}:=\min\{\Omega_{1},\Omega_{2}\}$.
\end{proof}

We stress the fact that Corollary~\ref{cor-3.2} applies to an \textit{arbitrary} nonconstant periodic
function of class $C^2$ provided that its period $T_{\Omega}$ is very large and its displacement from
a constant value $k > 0$ is very small.

As a next step, we plan to examine the case in which some form of
chaotic behavior can occur in situations when the forcing term has a sufficiently large period. On the other hand, we
will not assume any restriction regarding the smallness of the displacement. To this purpose, we
consider a different topological approach that comes from the Conley index theory and has been
considered by  Gedeon, Kokubu, Mischaikow and Oka in \cite{GKMO-00} to prove the
presence of chaotic solutions in slowly varying Hamiltonian systems. The method in \cite{GKMO-00} is stable
for small perturbations and applies also to systems which are not necessarily periodic in the time variable. Here,
we give an application to system
\begin{equation}\label{eq-3.1.1}
\begin{cases}
x'=y,\\
y'=-f(x)+p(\varepsilon t),
\end{cases}
\end{equation}
where $p: {\mathbb R} \to {\mathbb R}$ is a non-constant
periodic function of class $C^2$ such that $p(t) > 0$ for all $t\in {\mathbb R}.$
First we need to introduce a few definitions from \cite{GKMO-00}. Writing \eqref{eq-3.1.1} as
\begin{equation}\label{eq-3.1.1b}
\begin{cases}
x'=y,\\
y'=-f(x)+p(\theta),\\
\theta' = \varepsilon,
\end{cases}
\end{equation}
we set, for a moment, $\theta$ as a constant parameter and consider the planar autonomous Hamiltonian system
\begin{equation}\label{eq-3.1.1a}
\begin{cases}
x'=y,\\
y'=-f(x)+p(\theta).
\end{cases}
\end{equation}
Concerning this latter system, for each $\theta$ there exist an equilibrium point $A(\theta):=
(x_s(\theta),0)$ which is a center and a hyperbolic saddle equilibrium point $B(\theta):=(x_u(\theta),0)$ with a homoclinic orbit
enclosing $A(\theta).$ By definition, $f(x_u(\theta))= f(x_s(\theta)) = 0,$ with $x_u(\theta) < 0 < x_s(\theta).$
We denote also with ${\mathscr{A}}$ the set of all the points $(A(\theta), \theta)$ which is a curve of ${\mathbb R}^3.$
A solution $X(t):= (x(t),y(t),\theta(t))$ of system \eqref{eq-3.1.1b} is
said to oscillate $k$ times over an interval $I= [\theta^-,\theta^+]$ with respect to ${\mathscr{A}},$ if $k\in {\mathbb N}$
identifies the homotopy class of the closed loop
\begin{equation*}
\left(\bigcup_{\theta(t)\in I} X(t)\cup B(\theta(t))\right)\,\cup\,
\left(\bigcup_{\theta(t)\in \partial I} \overline{X(t)B(\theta(t))}\,\right)
\end{equation*}
in the fundamental group of ${\mathbb R}^3\setminus{\mathscr{A}}$ (isomorphic to $\mathbb{Z}$). Then, the results in
\cite{GKMO-00, KoMiOk-96}, applied to system \eqref{eq-3.1.1}, give the following conclusion.

\begin{theorem}\label{th-3.4} Assume $(H_{1})$ and let also $p: {\mathbb R} \to {\mathbb R}$ be a non-constant
periodic function of class $C^2$ such that $p(t)>0$ for all $t\in {\mathbb R}.$ Then there exists a choice of infinitely
may pairwise disjoint closed intervals $I_i :=[\theta_i^-,\theta_i^+]$ with
\begin{equation*}
\dots \theta_{i-1}^- < \theta_{i-1}^+ < \theta_{i}^- < \theta_{i}^+ < \theta_{i+1}^- < \theta_{i+1}^+  \dots, \quad i\in {\mathbb Z},
\end{equation*}
with the following property: for any given positive integer $K$ there exists $\bar{\varepsilon} > 0$ such that for any $\varepsilon$ with
$0 < \varepsilon < \bar{\varepsilon},$ there are at least two non-negative integers $m'_{i}$ and $m''_{i}$ (for $i$ odd)
and at least $K$ non-negative integers $m^{1}_{i},\dots m^{K}_{i}$ (for $i$ even), such that for each
sequence $(s_i)_{i\in {\mathbb Z}}$ of integers with $s_{2i+1} \in \{m'_{2i+1}, m''_{2i+1}\}$
and $s_{2i} \in \{m^{1}_{2i}, \dots, m^{K}_{2i}\}$, there is at least one solutions of \eqref{eq-3.1.1b} which oscillates
$s_i$ times over $I_i$.
\end{theorem}

\begin{proof}
The result follows from \cite[Cor.~1.2]{GKMO-00}, by observing that \eqref{eq-3.1.1} is a periodically perturbed
planar Hamiltonian system of the form $z' = J\nabla H(z,\varepsilon t)$, where $J$ is the $2\times 2$ symplectic matrix.
Without entering into discussion of technical details, we just give a list of the key points that make the setting of \cite{GKMO-00, KoMiOk-96}
applicable to our case.
We denote by $S(\theta)$ the area of the planar region containing the elliptic equilibrium point $A(\theta)$ and
bounded by the homoclinc orbit of \eqref{eq-3.1.1a} enclosing it. Then, the intervals $I_i$ are chosen
so that $S'(\theta_{i}^-) > 0 > S'(\theta_{i}^+)$ for $i$ odd and $S'(\theta_{i}^-) < 0 < S'(\theta_{i}^+)$ for $i$ even.
As a final remark, note that the method in \cite{GKMO-00} applies also when the forcing terms are not necessarily periodic
and, in this case, it is required the additional condition that $\theta^{-}_{i+1} - \theta^{+}_i$ is uniformly bounded away from zero.
However, in our situation, $p(t)$ is a nonconstant periodic function and, by denoting its fundamental period by $T$,
we can choose the intervals $I_{i}$ such that $\theta^{\pm}_{i+2} = \theta^{\pm}_{i} + T$ for all $i\in {\mathbb Z}$,
without any further hypothesis.
\end{proof}

It goes without saying that one of the main features of Theorem~\ref{th-3.4} regards the periodic perturbation $p(t)$ which is no longer required to be small,
as shown in Fig.~\ref{fig-2}.

\begin{figure}[htb]
          \centering
          \includegraphics[width=0.7\textwidth]{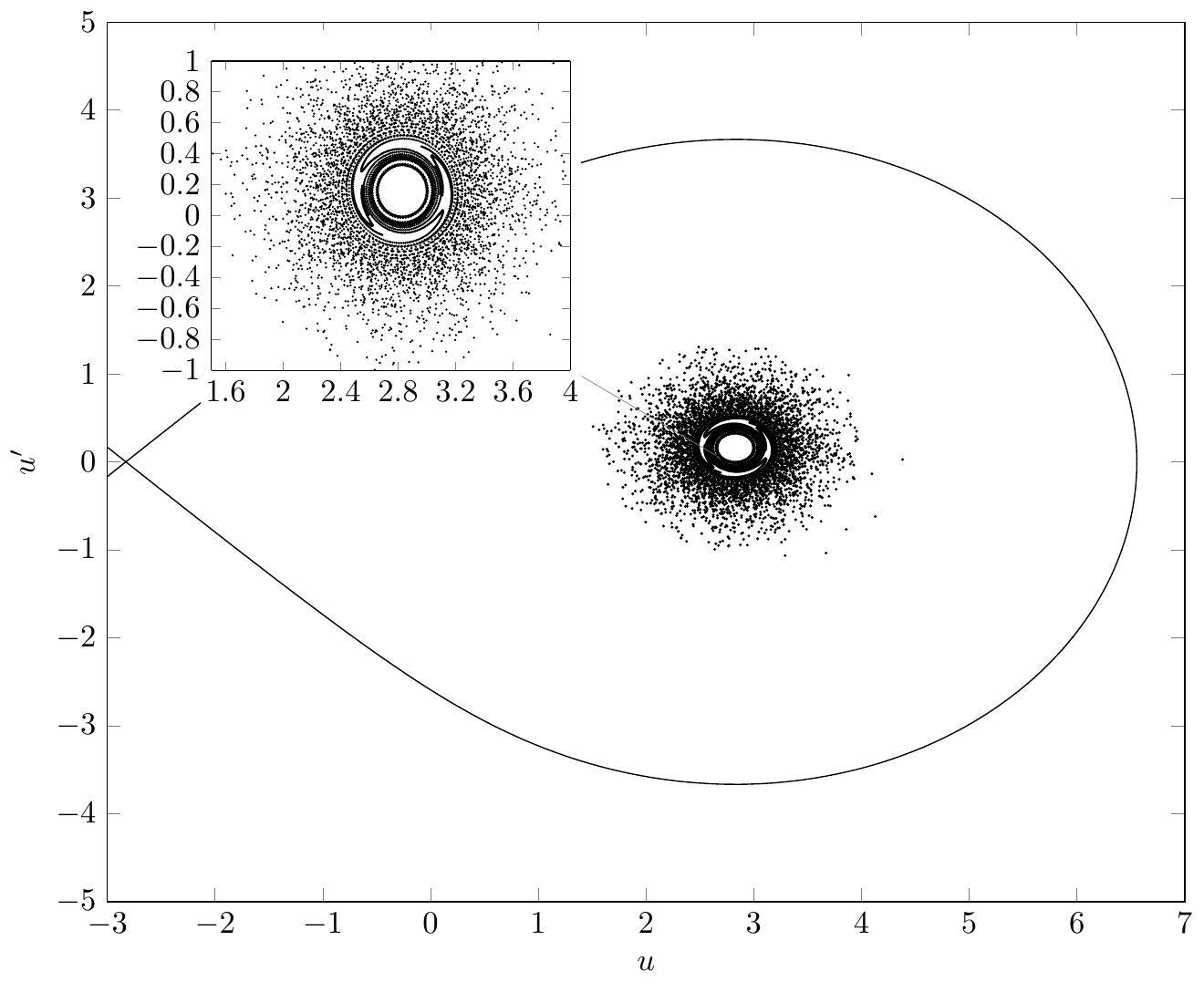}
\caption{Evolution of $u''+\sqrt{1+u^{2}} -1= 2+\varepsilon\sin(\omega t)$
in terms of the iterates of the Poincar\'e map, with $\varepsilon=1.5$, $\omega=0.1$
and varying 500 initial conditions $(u(0),u'(0))$ where $u(0)$ is within the interval $[-4,6]$ and $u'(0)=0$.
The continuous line denotes the homoclinic orbit $\mathcal{H}(x_{u})$ of the associated autonomous system \eqref{eq-2.2}.
Again the typical alternation of zones of stability and instability/randomness appears in a smaller region due to the fact that
only few initial conditions give rise to solutions bounded in the future.}
\label{fig-2}
\end{figure}

At last, as we mentioned above, the result achieved in \cite{GKMO-00} is stable under small perturbations.
This means that, in such a framework, Theorem~\ref{th-3.4} applies to system
\begin{equation*}
\begin{cases}
x'=y,\\
y'=- \varepsilon^2 c_0 y -f(x)+p(\varepsilon t),
\end{cases}
\end{equation*}
with $c_0 \in {\mathbb R}.$

Finally, we should recall some related theorems about the existence of multiple multi-bump homoclinics
which are based on an abstract variational perturbative approach, developed by Ambrosetti and Badiale in \cite{AmBa-98}
(see also \cite{AmBe-98, BeBo-99}). These theorems have found a natural application to
differential systems which are perturbations of autonomous
Schr\"{o}dinger type equations of the form $u'' - u + \nabla V (u) = 0,$ where, roughly $V(u) \simeq |u|^m$ for $ m > 2$
(cf.~\cite{Am-98}). Note that, in the one-dimensional case, the phase-plane geometry of these autonomous equations
is qualitatively similar to the one of \eqref{eq-2.1}.

\medskip\noindent
To conclude this section, in order to fix the ideas, let's recapitulate the relevant cases which can be tackled using Melnikov type results.
As a model to compare Corollary~\ref{cor-3.1}, Corollary~\ref{cor-3.2} and Theorem~\ref{th-3.4}, we take the system
\begin{equation}\label{eq-3.1.2}
\begin{cases}
x'=y,\\
y'=-f(x)+k + \varepsilon p_0(\omega t),
\end{cases}
\end{equation}
where we suppose that $k > 0$, $f$ is a $C^2$-function
satisfying $(H1)$ and $p_0(t)$ is a nonconstant $2\pi$-periodic function.
Thus the list of the consequence, that have been achieved by now, is the following.

\begin{itemize}
\item{Corollary~\ref{cor-3.1}: for a special form of $p_0(t)$ such as $p_0(t) = \cos(t + \vartheta_0)$, $\omega > 0$ arbitrary and $\varepsilon > 0$
small, we have found chaos in the sense of the Smale horseshoe for a suitable iterate of the Poincar\'{e} map $\Psi$ associated with system \eqref{eq-3.1.2}.}
\item{Corollary~\ref{cor-3.2}: for $p_0(t)$ an arbitrary function, $\omega > 0$ small and $\varepsilon > 0$
small (depending on the choice of $\omega$), we have found chaos in the sense of the topological horseshoe
for a suitable iterate of the Poincar\'{e} map $\Psi$ associated with system \eqref{eq-3.1.2} relatively to the period interval $[0,T]$ with $T:= 2\pi/\omega.$}
\item{Theorem~\ref{th-3.4}: for $p_0(t)$ an arbitrary function, $\omega > 0$ small and $|\varepsilon| < k$
(hence also a large forcing term with range in $]0,+\infty[$ is allowed), we have found symbolic dynamics for the Poincar\'{e} map $\Psi$
associated with system \eqref{eq-3.1.2} relatively to the period interval $[0,T]$ with $T:= 2\pi/\omega.$}
\end{itemize}
The last two results consider the case of the so-called slowly varying perturbations. Typically, these results guarantee the
presence of complex dynamics when a parameter tends to infinity (like the period $T$) or tends to zero (like the
coefficient $\omega$ in $p_0(\omega t)$).
The succeeding step is to provide specific bounds (for instance, lower bounds for the period)
in terms of the associated autonomous system. This will be carried out in the next section.

\subsubsection{Switched systems}\label{section-3.2}
Let us consider now a periodic piecewise constant forcing term which takes two values as follows
\begin{equation}\label{eq-3.2.1}
p_{k_1,k_2}(t):=\begin{cases}
k_{1}&\text{for }t\in[0,t_{1}[,\\
k_{2}&\text{for }t\in[t_{1},t_{1}+t_{2}[,
\end{cases}
\end{equation}
with $k_1, k_2 \geq 0$, $k_1\neq k_2$ and $t_{1},t_{2}>0$. We will perform our analysis by assuming
\begin{equation*}
0 < k_1 < k_2
\end{equation*}
and only briefly sketch how to deal when $k_1 = 0.$ We also suppose that the fundamental period of $p(t)$ splits as
\begin{equation*}
T:= t_1 + t_2.
\end{equation*}
In this setting, system $(\mathscr{S}_{1})$ is equivalent to the switched system in the phase plane $(x,y)$ which alternates between two subsystems:
\[\begin{cases}
x'=y,\\
y'=-f(x)+k_{i},
\end{cases}
\leqno{(\mathrm{S}_{i})}\]
for $i\in\{1,2\}$.
In other words, the solution to $(\mathscr{S}_{1})$ which starts from an initial point $z_{0}=(x_{0},y_{0})$
is governed by the subsystem $(\mathrm{S}_{1})$ for a fixed period of time $t_{1}$
and then it is governed by the subsystem $(\mathrm{S}_{2})$ for another fixed period of time $t_{2}$.
At this point, the switched system may change back to subsystem $(\mathrm{S}_{1})$ until the time elapsed is exactly $t_{1}+t_{2}$.
As a consequence, the Poincar\'e map $\Psi$ of system $(\mathscr{S}_{1})$ can be decomposed as $\Psi=\Psi_{2}\circ\Psi_{1}$,
where $\Psi_{i}$ is the Poincar\'e map of system $(\mathrm{S}_{i})$ relatively to the time interval $[0,t_{i}]$, for $i\in\{1,2\}$.

This particular choice of the forcing term is convenient to take advantage of the SAP method that
is collected in the Appendix (see also \cite{MPZ-09,PaZa-04} for the details).
Furthermore, switched systems are themselves an attractive topic in the field of control theory (see \cite{Ba-14}).
Therefore, by considering switching systems, we are looking for a geometry similar to the one of the ``linked twist maps'' (LTM)
(see \cite{PaZa-09,WiOt-04}). Specifically, the configuration of the problem we are analyzing recalls the work \cite{PPZ-08},
where the interplay between an annulus and a strip, instead of the usual two annuli, was discussed. With these remarks in mind, we are ready to prove the following result.

\begin{theorem}\label{th-3.5}
Assume $(H_{0})$ and let also $p: {\mathbb R} \to {\mathbb R}$ be a $T$-periodic stepwise function, such that $p(t)>0$ for all $t\in {\mathbb R}.$
Then, there exist $\tau^{*}_{1}$ and $\tau^{*}_{2}$ such that a topological horseshoe occurs
for the Poincar\'e map associated with system $(\mathscr{S}_{1})$ provided that $t_{1}>\tau^{*}_{1}$ and $t_{2}>\tau^{*}_{2}$.
\end{theorem}

Notice that $\tau^{*}_{1}$ and $\tau^{*}_{2}$ will be explicitly computed in terms of the forcing term $p(t)$ (cf.~\eqref{eq-tau1} and \eqref{eq-tau2}).

\begin{proof}
Consider two fixed values $k_{1},k_{2}$ and let $p(t)= p_{k_1,k_2}(t)$ be defined as in \eqref{eq-3.2.1}. In order to work with the SAP method, namely Theorem~\ref{th-app1},
we have to find two oriented topological rectangles $\tilde{\mathcal{M}}$ and $\tilde{\mathcal{N}}$ where chaotic dynamics take place.
Then, the analysis will be performed according to these steps which collect the stretching properties.

\begin{description}
\item[Step I] For any path $\gamma$ contained in $\mathcal{M}$, connecting the two sides $\mathcal{M}^{-}_{l}$ and $\mathcal{M}^{-}_{r}$,
there exist two sub-paths $\gamma_{0},\gamma_{1}$ such that $\Psi_{1}(\gamma_{i})$ is a path contained in $\mathcal{N}$ which joins
the two sides $\mathcal{N}^{-}_{l}$ and $\mathcal{N}^{-}_{r}$ for each $i\in\{0,1\}$.
\item[Step II] For any path $\gamma$ contained in $\mathcal{N}$, connecting the two sides $\mathcal{N}^{-}_{l}$ and $\mathcal{N}^{-}_{r}$,
there exist $m\geq 2$ sub-paths $\gamma_{0},\dots,\gamma_{m-1}$ such that $\Psi_{2}(\gamma_{i})$ is a path contained in $\mathcal{M}$
which joins the two sides $\mathcal{M}^{-}_{l}$ and $\mathcal{M}^{-}_{r}$ for each $i\in\{0,\dots,m-1\}$.
\end{description}

\medskip\noindent
We start by giving a suitable construction of these topological oriented rectangles. From Section~\ref{section-2}
follows the existence of two homoclinc orbits $\mathcal{H}(x_{u}(k_{1}))$ and $\mathcal{H}(x_{u}(k_{2}))$,
one for system $(\mathrm{S}_1)$ and one for $(\mathrm{S}_2)$, associated with the energies $\Phi_{k_{1}}(x_{u}(k_{1}))$ and $\Phi_{k_{2}}(x_{u}(k_{2}))$, respectively.
Moreover, Proposition~\ref{prop-2.3} leads to
\[x_{u}(k_{2})<x_{u}(k_{1})<x_{h}(k_{1})<x_{h}(k_{2}),
\]
which is equivalent to said that the region bounded by the homoclinic orbit $\mathcal{H}(x_{u}(k_{2}))$ contains the homoclinc orbit $\mathcal{H}(x_{u}(k_{1}))$.

Let us fix three main energy levels $A,B,D\in\mathbb{R}$ as follows.
Take $A<\Phi_{k_{1}}(x_{u}(k_{1}))$ such that the solution $a:=x_{*}(A)$ of the equation $\Phi_{k_{1}}(x)=A$ belongs to $]x_{u}(k_{2}),x_{u}(k_{1})[$.
Choose $\Phi_{k_{2}}(x_{s}(k_{2}))<D<\Phi_{k_{2}}(x_{u}(k_{2}))$ in a way that the solutions $d:=x_{-}(D)$
and $x_{+}(D)$ of the equation $\Phi_{k_{2}}(x)=D$ are such that $x_{u}(k_{2})<d<a$ and $x_{s}(k_{2})<x_{+}(D)<x_{h}(k_{2})$.
At last, consider $B>\Phi_{k_{1}}(x_{u}(k_{1}))$ so that the solution $b:=x^{*}(B)$ of $\Phi_{k_{1}}(x)=B$ is such that $x_{h}(k_{1})<b<x^{+}(D).$
In this way, one can determine three different energy level lines which are $\mathcal{V}_{A}$, $\mathcal{U}_{B}$
for system $(\mathrm{S}_1)$ and $\mathcal{O}_{D}$ for $(\mathrm{S}_2)$, defined
as in \eqref{eq-2.5v},
\eqref{eq-2.5u} and \eqref{eq-2.5o}, respectively. Now, we consider the closed regions
\[\begin{split}
&\mathcal{S}_{A}:= \{(x,y)\in\mathbb{R}^{2} : A\leq E_{k_{1}}(x,y)\leq \Phi_{k_{1}}(x_{u}(k_{1})), \, x\leq x_{u}(k_{1}) \},\\
&\mathcal{S}_{B}:= \{(x,y)\in\mathbb{R}^{2} : \Phi_{k_{1}}(x_{u}(k_{1}))\leq E_{k_{1}}(x,y)\leq B \},
\end{split}\]
and their union
\[
\mathcal{S}:= \mathcal{S}_{A} \cup \mathcal{S}_{B}\,.
\]
They are all invariant for the flow associated with system $(\mathrm{S}_1).$ The region $\mathcal{S}$
is topologically like a strip with a hole given by the part of the plane enclosed by the homoclinic trajectory
$\mathcal{H}(x_{u}(k_{1}))$. We also introduce a closed and invariant annular region for system $(\mathrm{S}_2)$, given by
\[
\mathcal{A}:= \{(x,y)\in\mathbb{R}^{2}: D\leq E_{k_{2}}(x,y)\leq \Phi_{k_{2}}(x_{u}(k_{2}))\}.
\]
The intersection of $\mathcal{S}$ with $\mathcal{A}$ determines two disjoint compact sets
that we call $\mathcal{M}$ (the one in the upper half-plane)
and $\mathcal{N}$ (the other symmetric one in the lower half-plane), that are
\[
\begin{split}
&\mathcal{M}:= \mathcal{A}\cap \mathcal{S}\cap \{(x,y)\in\mathbb{R}^{2}: y>0   \}, \\
&\mathcal{N}:= \mathcal{A}\cap \mathcal{S}\cap \{(x,y)\in\mathbb{R}^{2}: y<0   \}.
\end{split}
\]
One can easily check that they are topological rectangles. At last, we give an orientation as follows
\begin{align*}
\mathcal{M}^-_{l}&:= \mathcal{M}\cap \mathcal{V}_{A}\,, &  \mathcal{M}^-_{r}&:= \mathcal{M}\cap \mathcal{U}_{B}, \\
\mathcal{N}^-_{l}&:= \mathcal{N}\cap \mathcal{O}_{D}\,,  &  \mathcal{N}^-_{r}&:= \mathcal{N}\cap \mathcal{H}(x_{u}(k_{2})).
\end{align*}
See Fig.~\ref{fig-3} for a graphical sketch of $\tilde{\mathcal{M}}$ and $\tilde{\mathcal{N}}$.

\begin{figure}[htb]
          \centering
          \includegraphics[width=1\textwidth]{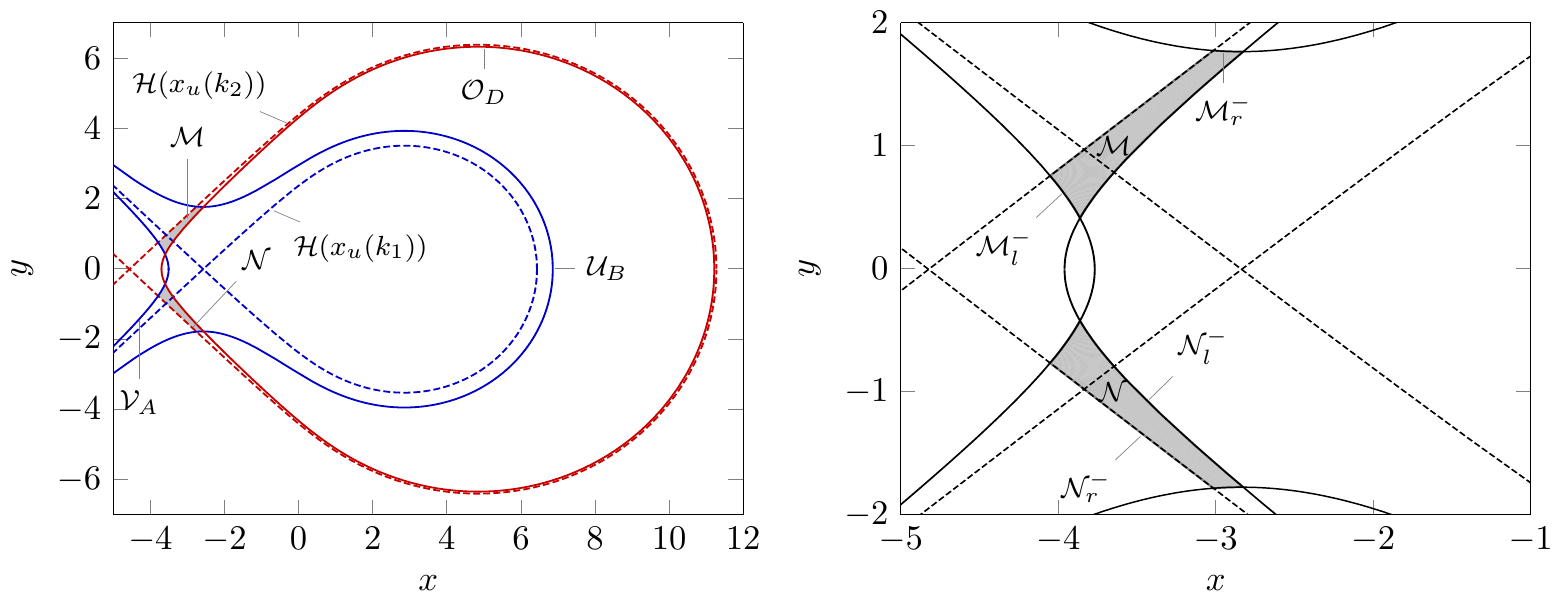}
          \caption{\small Left: Link of an annulus with a ``strip with hole''. Energy level lines for  system \eqref{eq-2.2} where $f(u)=\sqrt{1+x^{2}}-1$ are displayed with $k=2$ (blue) and $k=4$ (red). Right: Zooming of the topological rectangles with evidence of the boundaries.}
\label{fig-3}
\end{figure}

We are now in position to prove Step~I. Let us consider in this case system~$(\mathrm{S}_{1})$.
Then, thanks to the analysis performed in Section~\ref{section-2},
we known the time needed to move from the point $\big(x_{u}(k_{2}),\sqrt{2(A-\Phi_{k_{1}}(x_{u}(k_{2})))}\big)$
to the point $\big(x_{u}(k_{2}),-\sqrt{2(A-\Phi_{k_{1}}(x_{u}(k_{2})))}\big)$ along $\mathcal{V}_{A}$.
This is, in accord with \eqref{tau-v},
\[\tau_{\scriptscriptstyle\mathcal{V}_{A}}:=\tau_{\scriptscriptstyle\mathcal{V}}(A;\, x_{u}(k_{2})).\]
Instead, from \eqref{tau-u}, the move from the point $\big(x_{u}(k_{2}),\sqrt{2(B-\Phi_{k_{1}}(x_{u}(k_{2})))}\big)$
to the point $\big(x_{u}(k_{2}),-\sqrt{2(B-\Phi_{k_{1}}(x_{u}(k_{2})))}\big)$ along $\mathcal{U}_{B}$ requires the following time
\[\tau_{\scriptscriptstyle \mathcal{U}_{B}}:=\tau_{\scriptscriptstyle \mathcal{U}}(B;\, x_{u}(k_{2})).\]
As a result of these computations, we fix
\begin{equation}\label{eq-tau1}
\tau_{1}^{*}:=\max\{\tau_{\scriptscriptstyle\mathcal{V}_{A}},\tau_{\scriptscriptstyle\mathcal{U}_{B}}\}.
\end{equation}

Note that each solution of a Cauchy problem with initial conditions taken in
$\mathcal{M}$ evolves, through the action of $(\mathrm{S}_{1})$, inside the invariant region $\mathcal{S}$.
More in detail, at any time $t_1 > \tau_{\scriptscriptstyle\mathcal{V}_{A}}$, all the initial points in $\mathcal{M}^-_{l}$ will be moved,
along the level line $\mathcal{V}_{A}$, to points with $x < x_{u}(k_{2})$ and $y < 0$ by the action of $\Psi_1$.
Any solution $u(t)$ of $u'' + f(u) = k_1$ with $(u(0),u'(0))\in \mathcal{M}^-_{l}$
starts with $u(0) >  x_{u}(k_{2})$ and a positive slope, it is strictly increasing until it reaches its maximum value
$u_{\max} = a$ and then it decreases strictly till to the value $u(t_{1}) < x_{u}(k_{2})$. Moreover, $u'(t)$ is strictly
decreasing on the whole interval $[0,t_{1}].$ Similarly, for $t_{1} > \tau_{\scriptscriptstyle\mathcal{U}_{B}}$,
all the initial points in $\mathcal{M}^{-}_{r}$ will be moved away along the level line $\mathcal{U}_{B}$. The final points
will be such that $x < x_{u}(k_{2})$ and $y < 0.$ Analogous considerations can be made for the solution
$u(t)$ of $u'' + f(u) = k_1$ with $(u(0),u'(0))\in \mathcal{M}^{-}_{r}$ which achieves the maximum value
$u_{\max} = b$. In the region $\mathcal{M}$, any path connecting $\mathcal{M}^-_{l}$ to $\mathcal{M}^-_{r}$
must intersect the stable manifold $W^{s}(x_{u}(k_{1})).$ Notice that any solution $(x(t),y(t))$
of system $(\mathrm{S}_1)$ starting at a point of $W^{s}(x_{u}(k_{1}))$, lies on such a manifold and, therefore,
$y(t) > 0$ for all $t\geq 0.$

Let $\gamma:[0,1]\to\mathcal{M}$ be a continuous path with $\gamma(0)\in\mathcal{M}^{-}_{l}\subseteq \mathcal{V}_{A}$ and
$\gamma(1)\in\mathcal{M}^{-}_{r}\subseteq \mathcal{U}_{B}$.
First of all, observe that, by the continuity of $\gamma$ there exists $\bar{s}^{\flat},\bar{s}^{\sharp} \in]0,1[$
with $\bar{s}^{\flat} \leq \bar{s}^{\sharp}$
such that $\gamma(\bar{s}^{\flat}), \gamma(\bar{s}^{\sharp})\in W^{s}(x_{u}(k_{1}))$ and $\gamma(s) \in \mathcal{S}_{A}$ for all $0\leq s \leq \bar{s}^{\flat},$
as well as
$\gamma(s) \in \mathcal{S}_{B}$ for all $\bar{s}^{\sharp}\leq s \leq 1$.
By the choice of $\tau_{1}^{*}$, for each $t_{1}>\tau_{1}^{*}$ it
follows that
\[\begin{split}
&\Psi_{1}(\gamma(0)),\, \Psi_{1}(\gamma(1)) \in\{(x,y): x < x_{u}(k_{2}))),\, y < 0\},\\
&\Psi_{1}(\gamma(\bar{s}^{\flat})),\, \Psi_{1}(\gamma(\bar{s}^{\sharp})) \in\{(x,y): x > x_{u}(k_{2}))),\, y > 0\}.
\end{split}\]
Thus, the path $\gamma$ is folded onto itself in the invariant region
$\mathcal{S}$ by the action of system $\mathrm{S}_{1}$ as shown in Fig.~\ref{fig-4}.

\begin{figure}[htb]
          \centering
          \includegraphics[width=1\textwidth]{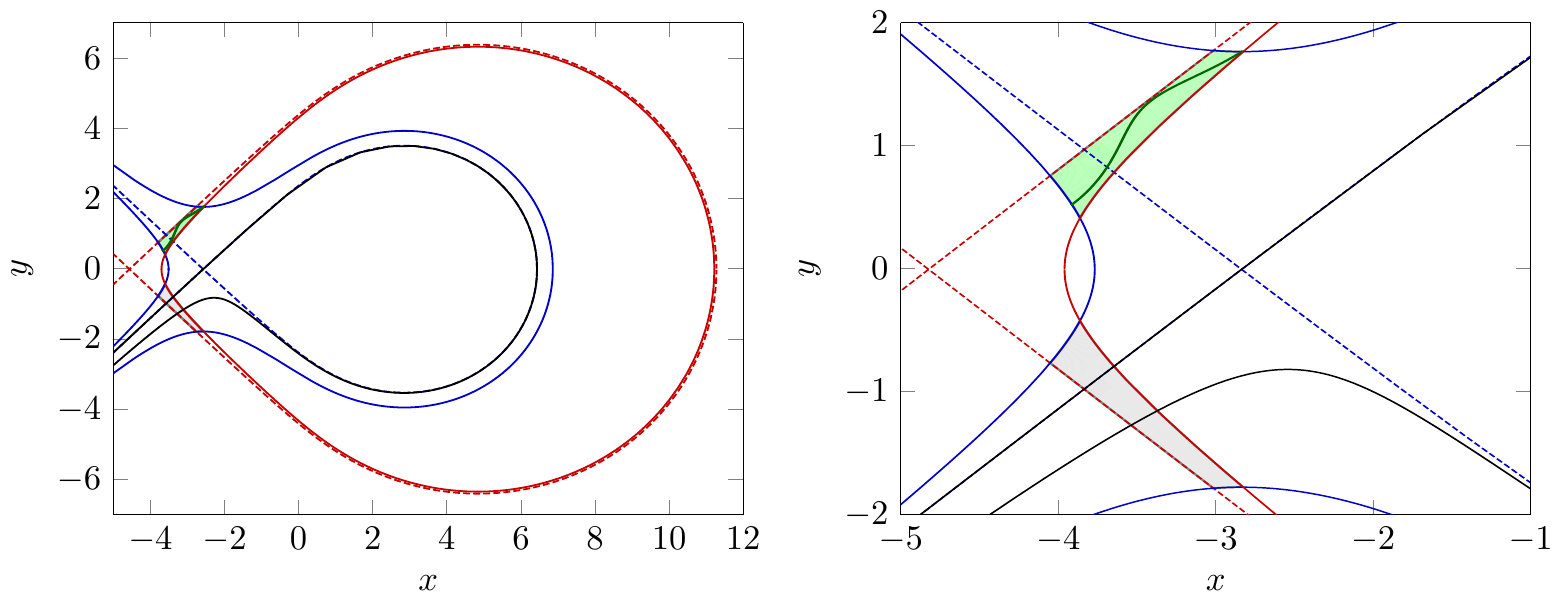}
          \caption{\small Left: Representation of a generic path $\gamma$ (green) in the topological rectangle $\mathcal{M}$ joining $\mathcal{M}^{-}_{l}$ with $\mathcal{M}^{-}_{r}$
          and its image (black) at time $t_{1}$ under the action of the system $(\mathrm{S}_{1})$ where $f(u)=\sqrt{1+x^{2}}-1$ and $k=2$.
          Right: Zooming of the two crossings between the image of the curve $\gamma$ with the topological rectangle $\mathcal{N}$.}
\label{fig-4}
\end{figure}

Now we set
\[\begin{split}
& s''_{A}:=\max\{s\in [0,\bar{s}^{\flat}]: \Psi_{1}(\gamma(s))\in \mathcal{N}^{-}_{l}\},\\
& s'_{A}:=\max\{s\in [0,s''_{A}]: \Psi_{1}(\gamma(s))\in \mathcal{N}^{-}_{r}\}.
\end{split}\]
By definition, $\gamma(s) \in \mathcal{M}\cap \mathcal{S}_{A}$
and  $\Psi_{1}(\gamma(s)) \in \mathcal{N}$ for all $s\in [s'_{A},s''_{A}]$
with $\Psi_{1}(\gamma(s'_{A}))\in \mathcal{N}^{-}_{r}$ and $\Psi_{1}(\gamma(s''_{A}))\in \mathcal{N}^{-}_{l}$.
Analogously, we define
\[\begin{split}
& s''_{B}:=\max\{s\in [\bar{s}^{\sharp},1]: \Psi_1(\gamma(s))\in \mathcal{N}^{-}_{r}\},\\
& s'_{B}:=\max\{s\in [\bar{s}^{\sharp},s''_{B}]: \Psi_1(\gamma(s))\in \mathcal{N}^{-}_{l}\},
\end{split}\]
and we observe that $\gamma(s) \in \mathcal{M}\cap \mathcal{S}_{B}$
and  $\Psi_{1}(\gamma(s)) \in \mathcal{N}$ for all $s\in [s'_{B},s''_{B}]$
with $\Psi_{1}(\gamma(s'_{B}))\in \mathcal{N}^{-}_{l}$ and $\Psi_1(\gamma(s''_{B}))\in \mathcal{N}^{-}_{r}$.

For any $t_{1} > \tau^{*}_1$ fixed, using an elementary continuity argument,
we can determine a (small) open neighborhood $\mathscr{W}$ of $W^{s}(x_u(k_1))\cap {\mathcal M}$
such that $y(t) > 0$ for all $t\in [0,t_1],$ whenever $(x(0),y(0))\in \mathscr{W}$.
Thus, finally, if we define
\[
\mathcal{K}_{1,0}:= \mathcal{M}\cap \mathcal{S}_{A} \setminus \mathscr{W}, \quad
\mathcal{K}_{1,1}:= \mathcal{M}\cap \mathcal{S}_{B} \setminus \mathscr{W},
\]
then, in accord with Definition~\ref{def-SAP}, we have determined two disjoint compact sets such that satisfy the
SAP condition with crossing number 2:
\[(\mathcal{K}_{1,0},\Psi_2)\colon\tilde{\mathcal{M}}\mathrel{\Bumpeq\!\!\!\!\!\!\longrightarrow^{2}}\tilde{\mathcal{N}},
\quad (\mathcal{K}_{1,1},\Psi_2)\colon\tilde{\mathcal{M}}\mathrel{\Bumpeq\!\!\!\!\!\!\longrightarrow^{2}}\tilde{\mathcal{N}}.\]

At last we consider system $(\mathrm{S}_{2})$ and we prove the stretching property formulated in Step~II.
Note that each solution of a Cauchy problem with initial
conditions taken in $\mathcal{N}$ evolves through the action of
$(\mathrm{S}_{2})$ inside the annular region $\mathcal{A}$ which is invariant for the
associated flow. Once the point $(b,0)$ is fixed as a center for polar coordinates, if the time increases, then
all the points of $\mathcal{A}\setminus\{(x_{u}(k_2),0)\}$ move along the energy
level lines of $(\mathrm{S}_{2})$ in the clockwise sense. For our purposes, it will be convenient to
introduce an angular variable starting from the half-line $L:=\{(r,0): r < b\}$ and
counted positive clockwise from the reference axis $L.$ In this manner all the points of
$\mathcal{N}$ are determined by an angle $\vartheta \in\,]-\pi/2,0[$ (mod $2\pi$),
while those of $\mathcal{M}$ are determined by an angle $\vartheta \in\,]0,\pi/2[$ (mod $2\pi$).
In other words, for our auxiliary polar coordinate system, the region $\mathcal{N}$
(respectively, $\mathcal{M}$) lies in the interior of the fourth quadrant (respectively, first quadrant).
Any solution $u(t)$ of $u'' + f(u) = k_2$ with $(u(0),u'(0))\in \mathcal{N}^{-}_{r}$
starts with $u(0) >  x_{u}(k_{2})$ and a negative slope, it tends as $t\to +\infty$
to the saddle point of $(\mathrm{S}_2)$ along the homoclinic orbit, with $u(t)$ decreasing and $u'(t)$ increasing.
On the other hand, any solution with $(u(0),u'(0))\in \mathcal{N}^{-}_{l}$ is periodic with period equal to
the fundamental period of the orbit $\mathcal{O}_{D}$, that we denote by
\[
\mathcal{T}_{\scriptscriptstyle\mathcal{O}_{D}}:=2\, \tau_{\scriptscriptstyle\mathcal{O}}(D;\, x_{+}(D)),
\]
by means of \eqref{tau-o}.
If we take any path in $\mathcal{N}$ connecting ${\mathcal N}^{-}_{r}$ to ${\mathcal N}^{-}_{l}$
we have that its image under the action of the flow of $(\mathrm{S}_2)$ looks like a spiral curve
contained in $\mathcal{A}$ which winds a certain  number of times around the center. In order to
formally prove this fact and to evaluate the precise number of revolutions, we denote by
$\vartheta(t,z)$ the angle at the time $t\geq 0$ associated with the solution $(x(t),y(t))$ of
system $(\mathrm{S}_{2})$ such that $(x(0),y(0)) = z\in \mathcal{N}.$ By the previous considerations and the
choice of a clockwise orientation, we know that $\tfrac{d}{dt}\vartheta(t,z) > 0$ for all $z\in \mathcal{N}.$
For our next computations we need also to introduce the time needed to go
from the point $\big(b,-\sqrt{2(D-\Phi_{k_{2}}(b))}\big)$
to the point $\big(b,\sqrt{2(D-\Phi_{k_{2}}(b))}\big)$ along $\mathcal{O}_{D}$, which is given by
\[\tau_{\scriptscriptstyle\mathcal{O}_{D}}:=\tau_{\scriptscriptstyle\mathcal{O}}(D;\, b),\]
consistently with \eqref{tau-o}. Given $m\geq 1$, we fix
\begin{equation}\label{eq-tau2}
\tau_{2}^{*}:=\tau_{\scriptscriptstyle\mathcal{O}_{D}} + (m-1) \mathcal{T}_{\scriptscriptstyle\mathcal{O}_{D}}.
\end{equation}
We claim that for each fixed time $t_{2} >\tau_{2}^{*}$ the SAP property holds for the Poincar\'{e} map $\Psi_{2}$
with crossing number (at least) $m.$ A visualization of this step for $m=1$ is given in Figure \ref{fig-5}.

\begin{figure}[htb]
          \centering
          \includegraphics[width=1\textwidth]{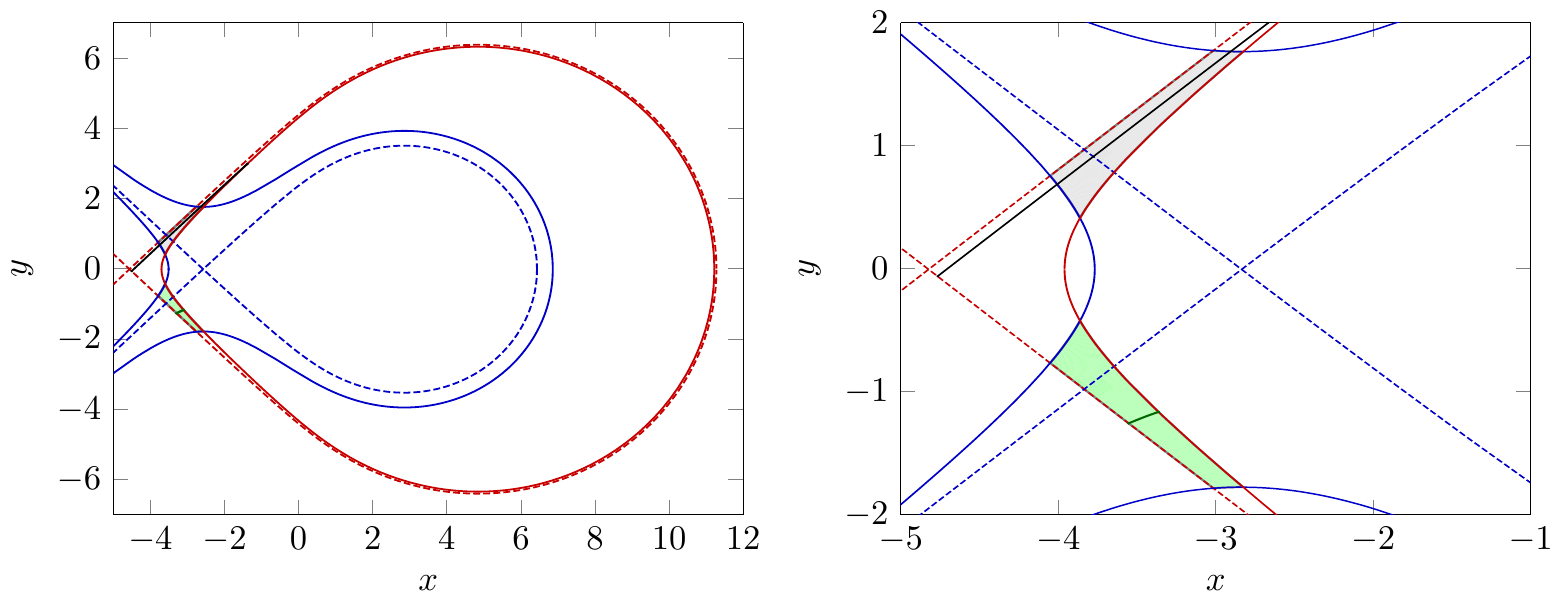}
          \caption{Left: Representation of a generic path $\gamma$ (green) in the topological rectangle $\mathcal{N}$ joining $\mathcal{N}^{-}_{r}$ with $\mathcal{N}^{-}_{l}$
          and its image (black) at time $t_{1}+t_{2}$ under the action of the system $(\mathrm{S}_{2})$ where $f(u)=\sqrt{1+x^{2}}-1$ and $k=4$.
          Right: Zooming of the crossing between the image of the curve $\gamma$ with the topological rectangle $\mathcal{M}$.}
\label{fig-5}
\end{figure}

By the previous observations, we have that
\[\begin{split}
&\vartheta(t_{2},z) < 0, \quad\forall\, z\in \mathcal{N}^{-}_{r},\\
&\vartheta(t_{2},z) >  \frac{\pi}{2} + 2(m-1)\pi, \quad\forall\, z\in \mathcal{N}^{-}_{l}.
\end{split}\]
This allows us to introduce $m$ nonempty subsets $\mathcal{K}_{2,0}\,,\dots, \mathcal{K}_{2,m-1}$
of $\mathcal{N}$ which are pairwise disjoint and compact. They are defined by
\[
\mathcal{K}_{2,i}:= \{z\in {\mathcal N}: \vartheta(t_{2},z) \in [2i\pi,(\pi/2) + 2i\pi]\}, \quad\forall i\in\{0,\dots, m-1\}.
\]

Let $\gamma:[0,1]\to\mathcal{N}$ be a continuous path with $\gamma(0)\in\mathcal{N}^{-}_{r}\subseteq \mathcal{H}(x_{u}(k_{2}))$ and
$\gamma(1)\in\mathcal{N}^{-}_{l}\subseteq \mathcal{O}_{D}$. We fix also an index $i\in \{0,\dots,m-1\}.$
First of all, observe that, by the continuity of $\gamma$ there exists $\bar{s}^{\flat}_{i},\bar{s}^{\sharp}_{i} \in]0,1[$
with $\bar{s}^{\flat}_{i} < \bar{s}^{\sharp}_{i}$ such that
\[\begin{split}
&\vartheta(t_{2},\gamma(\bar{s}^{\flat}_{i})) = 2i\pi,\\
&\vartheta(t_{2},\gamma(\bar{s}^{\sharp}_{i})) = \frac{\pi}{2} + 2i\pi,
\end{split}\]
and
\[
2i\pi < \vartheta(t_2,\gamma(s)) < \frac{\pi}{2} + 2i\pi, \quad\forall\,
\bar{s}^{\flat}_{i} < s < \bar{s}^{\sharp}_{i}\,.
\]
For ease of notation, we define as $\mathcal{A}_{\mathrm{I}}$
the intersection of $\mathcal{A}$ with the first quadrant of the auxiliary polar coordinate system
and we also consider the following two segments
\[\begin{split}
&\mathcal{A}_{\mathrm{I}}^{x}:=\left[x_{u}(k_{2}),d\right]\times\{0\},\\
&\mathcal{A}_{\mathrm{I}}^{y}:=\{b\}\times \left[\sqrt{2(D-\Phi_{k_{2}}(b))},\sqrt{2(\Phi_{k_2}(x_u(k_{2}))-\Phi_{k_{2}}(b))}\right],
\end{split}\]
which are on the boundary of $\mathcal{A}_{\mathrm{I}}$. By construction,
the image $\Psi_{2}\circ \gamma|_{[\bar{s}^{\flat}_{i}, \bar{s}^{\sharp}_{i}]}$ is contained in
 $\mathcal{A}_{\mathrm{I}}$ and joins  $\mathcal{A}_{\mathrm{I}}^x$ to $\mathcal{A}_{\mathrm{I}}^y.$ On the other hand,
 the set $\mathcal{M}$ as well as its sides $\mathcal{M}^{-}_{l}$ and ${\mathcal M}^-_{r}$
 separate $\mathcal{A}_{\mathrm{I}}^x$ and $\mathcal{A}_{\mathrm{I}}^y$ inside  $\mathcal{A}_{\mathrm{I}}.$
 An elementary connectivity argument, allows to determine $s'_{i}$ and $s''_{i}$ with
 $\bar{s}^{\flat}_{i} < s'_{i} < s''_{i} < \bar{s}^{\sharp}_{i}$ such that
$\Psi_{2}(\gamma(s'_{i})) \in \mathcal{M}^{-}_{l}$, $\Psi_{2}(\gamma(s''_{i})) \in \mathcal{M}^{-}_{r}$
and, $\Psi_{2}(\gamma(s)) \in \mathcal{M},$ for all $s\in [ s'_{i},s''_{i}].$
Moreover, $\gamma(s) \in \mathcal{K}_{2,i}$ for all $s\in [ s'_{i},s''_{i}].$
In this way our claim is verified because
\[(\mathcal{K}_{2,i},\Psi_2)\colon\tilde{\mathcal N}\mathrel{\Bumpeq\!\!\!\!\!\!\longrightarrow}\tilde{\mathcal M},
\quad\forall\, i\in\{0,\dots,m-1\}.\]

At the end, from Step~I and Step~II we can conclude that there exists a topological horseshoe for the Poincar\'{e} map
$\Psi = \Psi_{2}\circ \Psi_{1}$ with full dynamics on $2 \times m$ symbols.
\end{proof}

\begin{remark}
The above proof can be adapted to cover some slight variants of the theorem.
In particular, the following two observations can be made.
\begin{itemize}
\item[$(i)$\;]
We have proved Theorem~\ref{th-3.5} for a stepwise forcing term. However, the result still holds for a forcing term
$p(t)$ which can be also smooth, although near to $p_{k_1,k_2}(t)$
in the $L^1$-norm. In fact, such a result is stable with respect to small perturbations in the following sense:
for any choice of $t_{1}>\tau^{*}_{1}$ and $t_{2}>\tau^{*}_{2}$ (so that $T = t_1 + t_2$ is fixed)
there exists an $\varepsilon_{0}>0$ such that, for all $c$ with $|c|<\varepsilon_{0}$ and every
forcing term $p(t)$ such that $\int_{0}^{T}|p(t)-p_{k_{1},k_{2}}(t)| \,dt<\varepsilon_{0},$
the conclusion of Theorem~\ref{th-3.5} holds for system $(\mathscr{S}_2)$.
\item[$(ii)$\;]
In Theorem~\ref{th-3.5} we have assumed $0 < k_1 < k_2$. With a minor effort, one can easily adapt the proof also to the choice of $0 = k_1 < k_2$.
Notice that in this case, the associated dynamical system is described by a switch between two autonomous systems whose portraits are
shown in Fig.~\ref{fig-1}. The regions ${\mathcal M}$ and ${\mathcal N}$ are now determined by the intersection
of an annular region and a topological strip.
\end{itemize}
\end{remark}

As already mentioned, a peculiar aspect of our approach is due to the possibility to provide estimates from below for the period of the forcing
term, as given in \eqref{eq-tau1} and \eqref{eq-tau2}. In some cases, these constants can be easily determined and are not necessarily large.
To conclude the paper, an example in this direction is given. To simplify the treatment, we restrict ourselves
to consider the nonlinearity $f(s) = |s|.$
Clearly the function $f$ is not smooth, but since the result is stable for small perturbations,
our example can be adapted to the case of a smooth nonlinearity sufficiently near to the absolute value.

\begin{example}
Let us consider the second order ODE $u'' + |u| = p_{0,2}(t)$
for $T = t_1 + t_2$,
which is equivalent to the differential system
\begin{equation}\label{eq-ex1}
\begin{cases}
x' = y,\\
y' = - |x| +p_{0,2}(t),
\end{cases}
\end{equation}
where the forcing term $p_{0,2}(t)$ is defined according to \eqref{eq-3.2.1}.
Then, using essentially the same argument developed in order to prove Theorem~\ref{th-3.5}, symbolic dynamics on two symbols can be detected for system \eqref{eq-ex1}.
With this respect, we consider two regions in the phase plane defined as follows
\[\begin{split}
&\mathcal{S}:= \{(x,y)\in\mathbb{R}^{2} : 0\leq E_{0}(x,y)\leq 8 \},\\
&\mathcal{A}:= \{(x,y)\in\mathbb{R}^{2}: \rho_{\epsilon}\leq E_{2}(x,y)\leq 2\},
\end{split}\]
where $\rho_{\epsilon}:=(-\epsilon^{2}+4\epsilon)/2$ with $\epsilon>0$ a sufficiently small fixed real value.
The strip region $\mathcal{S}$ is obtained from the equation $u'' + |u| = 0$
by considering the area between the following associated level lines:
the unbounded orbit $\mathcal{U}_{8}$ passing through the point $(4, 0)$
and the line $x = -|y|$ made by the unstable equilibrium point $(x_{u}(0),0)=(0,0)$, the stable manifold $\mathcal{W}^{s}(0)$ and the unstable one $W^{u}(0)$.
To construct the annular region $\mathcal{A}$, we consider the equation $u'' + |u| = 2$ and from its phase portrait we select
the area between the homoclinic orbit $\mathcal{H}(-2)$ at the saddle point $(x_{u}(2),0)=(-2,0)$ and a periodic orbit $\mathcal{O}_{\rho_{\epsilon}}$
that passes through $(\epsilon,0)$ which is a point very close to the origin.
Dealing with $u'' + |u| = 2$ we can observe that all the periodic orbits enclosing the stable center $(2,0)$
and contained in the right-half phase plane are isochronous with period $2\pi$.
Now, we set the topological rectangles as follows
\[
\begin{split}
&\mathcal{M}:= \mathcal{A}\cap \mathcal{S}\cap \{(x,y)\in\mathbb{R}^{2}: y>0   \}, \\
&\mathcal{N}:= \mathcal{A}\cap \mathcal{S}\cap \{(x,y)\in\mathbb{R}^{2}: y<0   \},
\end{split}
\]
and the orientation is analogous to the one just given in the proof of the previous theorem.

To apply the SAP method we require the following time mapping estimates.
First, the time needed to move along $\mathcal{U}_{8}$ from the point $(2\sqrt{2},2\sqrt{2})$ to the point $(2\sqrt{2},-2\sqrt{2})$, which is
\[\tau_{\scriptscriptstyle\mathcal{U}_{8}}:=\tau_{\scriptscriptstyle\mathcal{U}}(8;\, 4)=2\int^{4}_{2\sqrt{2}}\frac{ds}{\sqrt{16-x^{2}}}=\frac{\pi}{2}.
\]
Next, the period $\tau_{\scriptscriptstyle\mathcal{O}_{\rho_{\epsilon}}}$ of the periodic orbit $\mathcal{O}_{\rho_{\epsilon}}$, which is
$\tau_{\scriptscriptstyle\mathcal{O}_{\rho_{\epsilon}}}\sim 2\pi$ when $\epsilon$ is chosen small enough.
In this way, by fixing $t_{1}>\pi/2$, the image at time $t=t_{1}$ of any continuous path contained in $\mathcal{M}$
which connects $\mathcal{M}^{-}_{l}$ to $\mathcal{M}^{-}_{r}$,
is stretched under the action of system \eqref{eq-ex1} in another continuous path and for it
one can find a sub-path entirely contained in $\mathcal{N}$ which connects $\mathcal{N}^{-}_{l}$ to $\mathcal{N}^{-}_{r}$.
Provided that $t_{2}>4\pi$, the previous sub-path is again stretched by system \eqref{eq-ex1}
and at time $t=T=t_{1}+t_{2}$ its image has revolved at least twice around the center $(2,0)$.
From this image, which is a spiral-like curve, we can detect two sub-paths in $\mathcal{M}$
that join the two sides $\mathcal{M}^{-}_{l}$ and $\mathcal{M}^{-}_{r}$.
In conclusion, if the period of the forcing term $p_{0,2}(t)$ is such that $T>9\pi/2$, then
Theorem~\ref{th-app1} guarantees dynamics on $1\times 2$ symbols for system \eqref{eq-ex1}, that is a topological horseshoe occurs.

This toy model is useful to determine the number of eigenvalues of the differential operator
$u\mapsto u''$ with $T$-periodic boundary conditions, namely $\lambda_{j}=(j-1)^{2}(2\pi/T)^{2}$ for $j=1,2,\dots$, crossed by the nonlinearity.
Since $f'(+\infty)=1$ it follows that the range where complex dynamics take place is between $\lambda_{3}$ and $\lambda_{4}$.
\end{example}

%------------------------------------------------------------------------
%--- Section appendix
\section{Appendix: SAP method}
In this section we summarize the topological tool, called \emph{stretching along the paths} (SAP) method,
which is the core of the application in Section~\ref{section-3.2}.
The appendix will be conform to the framework of the paper for the convenience of the reader.
We refer to \cite{MPZ-09,PaZa-04} for a detailed presentation of that theory.
First, we provide some basic notations and definitions.

\begin{definition}
Let $\mathcal{R}$ be a set homeomorphic to $[0,1]\times[0,1]$.
The pair $\tilde{\mathcal{R}}:=(\mathcal{R},\mathcal{R}^{-})$ is called \emph{oriented topological rectangle}
if $\mathcal{R}^{-}=\mathcal{R}^{-}_{l}\cup\mathcal{R}^{-}_{r}$,
where $\mathcal{R}^{-}_{l}$ and $\mathcal{R}^{-}_{r}$ are two disjoint compact arcs contained in $\partial\mathcal{R}.$
\end{definition}

\begin{definition}\label{def-SAP}
Given two topological oriented rectangles $\tilde{\mathcal{M}}:=(\mathcal{M},\mathcal{M}^{-})$,
$\tilde{\mathcal{N}}:=(\mathcal{N},\mathcal{N}^{-})$ and a continuous map $\phi:\mathrm{dom}\,\phi\subseteq\mathbb{R}^{2}\to\mathbb{R}^{2}$,
we say that $\phi$ \emph{stretches $\tilde{\mathcal{M}}$ to $\tilde{\mathcal{N}}$ along the paths}
and we write
\[(\mathcal{K},\phi)\colon\tilde{\mathcal M}\mathrel{\Bumpeq\!\!\!\!\!\!\longrightarrow}\tilde{\mathcal N}\]
if $\mathcal{K}$ is a compact subset of $\mathcal M\cap\mathrm{dom}\,\phi$
and for every path $\gamma\colon[0,1]\to\mathcal{M}$ such that $\gamma(0)\in\mathcal{M}^{-}_{l}$
and $\gamma(1)\in\mathcal{M}^{-}_{r}$ (or vice-versa), there exists a subinterval $[t',t'']\subseteq[0,1]$ such that
\begin{itemize}
\item $\gamma(t)\in \mathcal{K}$ for all $t\in[t',t'']$,
\item $\phi(\gamma(t))\in\mathcal{N}$ for all $t\in[t',t'']$,
\item $\phi(\gamma(t'))$ and $\phi(\gamma(t''))$ belong to different components of $\mathcal{N}^{-}$.
\end{itemize}

\noindent
Given a positive integer $m$, we say that $\phi$ \emph{stretches $\tilde{\mathcal{M}}$ to $\tilde{\mathcal{N}}$ along the paths with crossing number} $m$
and we write
\[(\mathcal{K},\phi)\colon\tilde{\mathcal M}\mathrel{\Bumpeq\!\!\!\!\!\!\longrightarrow^{m}}\tilde{\mathcal N}\]
if  there exist $m$ pairwise disjoint compact sets $\mathcal{K}_{0},\dots,\mathcal{K}_{m-1}\subseteq \mathcal M\cap\mathrm{dom}\,\phi$
such that $(\mathcal{K}_{i},\phi)\colon\tilde{\mathcal M}\mathrel{\Bumpeq\!\!\!\!\!\!\longrightarrow}\tilde{\mathcal N}$ for each $i\in\{0,\dots,m-1\}$.
\end{definition}

Clearly, the case $m\geq 2$ is the more interesting one in the applications because it lays the groundwork
for achieve a very rich symbolic dynamic structure.
Before linking the notion of ``stretching along the paths'' with the one of ``chaos''
we have to recall what do we mean when a topological horseshoe occurs.
Inspired by the definition of chaos in the sense of coin tossing or in the sense of Block-Coppel we introduce what follows.

\begin{definition}\label{def-app1}
Let $\phi:\mathrm{dom}\,\phi\subseteq\mathbb{R}^{2}\to\mathbb{R}^{2}$ be a map and
let $\mathcal{D}\subseteq\mathrm{dom}\,\phi$ be a nonempty set.
We say that \emph{$\phi$ induces chaotic dynamics on $m\geq 2$ symbols on a set $\mathcal{D}$}
if there exist $m$ nonempty pairwise disjoint compact sets
\[\mathcal{K}_{0}, \dots, \mathcal{K}_{m-1}\subseteq\mathcal{D}\]
such that for each two-sided sequence $(s_{i})_{i\in\mathbb{Z}}\in\{0,\dots,m-1\}^{\mathbb{Z}}$
there exists a corresponding sequence $(w_{i})_{i\in\mathbb{Z}}\in\mathcal{D}^{\mathbb{Z}}$ such that
\begin{equation}\label{eq-ap1}
w_{i}\in\mathcal{K}_{s_{i}} \text{ and } w_{i+1}=\phi(w_{i}) \text{ for all } i\in\mathbb{Z},
\end{equation}
and, whenever $(s_{i})_{i\in\mathbb{Z}}$ is a $k$-periodic sequence for some $k\geq1$ there exists a
$k$-periodic sequence $(w_{i})_{i\in\mathbb{Z}}\in\mathcal{D}^{\mathbb{Z}}$ satisfying \eqref{eq-ap1}.
\end{definition}

It is important to note that Definition~\ref{def-app1} has several relevant consequences which are discussed in \cite[Th.~2.2]{MPZ-09}
and in \cite{MRZ-10,PPZ-08}.
In particular, for a one-to-one map $\phi$, it ensures the existence of a nonempty compact invariant set
$\Lambda\subseteq\cup^{m-1}_{i=0}\mathcal{K}_{i}\subseteq\mathcal{D}$ such that
$\phi_{|\Lambda}$ is semiconjugate to the Bernoulli shift map on $m\geq 2$ symbols by a continuous surjection $\Pi$.
Moreover, it guarantees that the set of the periodic points of $\phi$ is dense in $\Lambda$ and,
for all two-sided periodic sequence $S\in\Sigma_{m}$, the preimage $\Pi^{-1}(S)$ contains a periodic point of $\phi$ with the same period.
In view of these properties, we recognize exactly the requirements needed
to assert that a topological horseshoe occurs (cf.~Introduction).

Finally, in order to detect chaos, an useful topological tool is the following result
which takes into account the particular nature of switched systems we deal with.

\begin{theorem}[\protect{\cite[Th.~2.1]{MRZ-10}}]\label{th-app1}
Let $\varphi:\mathrm{dom}\,\varphi\subseteq\mathbb{R}^{2}\to\mathbb{R}^{2}$ and
$\psi:\mathrm{dom}\,\psi\subseteq\mathbb{R}^{2}\to\mathbb{R}^{2}$ be continuous maps.
Let $\tilde{\mathcal{M}}=(\mathcal{M},\mathcal{M}^{-})$ and $\tilde{\mathcal{N}}=(\mathcal{N},\mathcal{N}^{-})$
be oriented rectangles in $\mathbb{R}^{2}$. Suppose that
\begin{itemize}
\item there exist $n\geq 1$ pairwise disjoint compact subsets of $\mathcal{M}\,\cap\,\mathrm{dom}\,\varphi$,
$\mathcal{Q}_{0},$ $\dots,$ $\mathcal{Q}_{n-1}$, such that
$(\mathcal{Q}_{i},\varphi)\colon\tilde{\mathcal M}\mathrel{\Bumpeq\!\!\!\!\!\!\longrightarrow}\tilde{\mathcal N}$ for $i=0,\dots,n-1$,
\item there exist $m\geq 1$ pairwise disjoint compact subsets of $\mathcal{N}\,\cap\,\mathrm{dom}\,\psi$, 
$\mathcal{K}_{0},$ $\dots,$ $\mathcal{K}_{m-1}$, such that
$(\mathcal{K}_{i},\psi)\colon\tilde{\mathcal N}\mathrel{\Bumpeq\!\!\!\!\!\!\longrightarrow}\tilde{\mathcal M}$ for $i=0,\dots,m-1$.
\end{itemize}
If at least one between $n$ and $m$ is greater or equal than $2$, then the map $\phi=\psi\circ\varphi$
induces chaotic dynamics on $n\times m$ symbols on
\[\mathcal{Q}^{*}=\bigcup_{\substack{i=1,\dots,n \\j=1,\dots,m}} \mathcal{Q}_{i}\cap\varphi^{-1}(\mathcal{K}_{j}).\]
\end{theorem}

We observe that the trick of the method is the verification of some stretching properties for the maps $\varphi$ and $\psi$. In our context, it applies to the Poincar\'e map $\Psi=\Psi_{2}\circ\Psi_{1}$ associated with system $(\mathscr{S}_{1})$ which is a homeomorphism of $\mathrm{dom}\,\Psi$ onto its image.

%--------------------------------------------
%--- Bibliography
\bibliographystyle{elsart-num-sort}
\bibliography{ReferenceSZ}

%\bigskip
%\begin{flushleft}
%
%{\small{\it Preprint}}
%
%{\small{\it June 2017}}
%
%\end{flushleft}
% ------------------------------------------------------------------------
\end{document}